\documentclass[12pt]{extarticle}
\usepackage[top=70pt,bottom=70pt,left=75pt,right=77pt]{geometry}

\usepackage{amssymb, amsmath, amsthm, setspace, bbm, prettyref, graphicx, float, subfigure, color, mathrsfs, mathtools,comment}

\usepackage{Lutsko_Style}
\DeclareMathSizes{10}{10}{7}{5}

\title{An Abstract Spectral Approach to Horospherical Equidistribution}
\author{Christopher Lutsko}

\begin{document}

  \maketitle
  \begin{abstract}
    \noindent This paper introduces an abstract spectral approach to prove effective equidistribution of expanding horospheres in hyperbolic manifolds. The method, which is motivated by the approach to counting developed by (Lax-Phillips 1982), produces highly effective, explicit error terms. To exhibit the flexibility of this method we prove effective horospherical equidistribution theorems in $T^1(\half^{n+1})$ and in the higher rank setting, $\SL_n(\R) / \SO_n(\R)$.
   
  \end{abstract}

  \onehalfspacing
  \setlength{\abovedisplayskip}{1mm}

  \section{Introduction}
  
  Given a manifold, a flow, and an expanding submanifold under the action of the flow, a key question in dynamical systems is to establish when such a submanifold equidistributes, and moreover, to establish effective rates of equidistribution. Establishing this behavior is at the heart of some of the crowning achievements of homogeneous dynamics (e.g Ratner's theorems \cite{Ratner1991a,Ratner1991b}, and horospherical equidistribution \cite{Zagier1981, Sarnak1981, OhShah2013}). Moreover, these equidistribution results are crucial to the application of dynamical methods to number theory problems. One such dynamical method, uses the Margulis thickening trick to pass from mixing to equidistribution statements. 

  While this dynamical method is very strong, it is in some sense wasteful. That is, there are some more specific questions for which spectral methods have proven to be more effective and flexible. In particular, for counting points in group orbits an abstract spectral method developed by Lax and Phillips \cite{LaxPhillips1982} (itself modeled on previous approaches of Huber, Patterson, and Selberg \cite{Huber1956, Patterson1975, Selberg1989}) was used to obtain effective asymptotic results beyond the reach of these dynamical equidistribution methods. Moreover, the abstract spectral approach is extremely 'soft', in that it requires very little specific input from the problem. Thus, the method is potentially extremely flexible (see \cite{Kontorovich2009, KontorovichLutsko2022} for extensions). 

  In this paper, we introduce a method for proving effective horospherical equidistribution which is based on the abstract spectral method to counting. To exhibit this method we will recover and improve on two existing effective horospherical equidistribution theorems. In the rank one setting we recover (with an improved error rate) an effective equidistribution theorem of Mohammadi and Oh \cite{MohammadiOh2015} for expanding horospheres in $\Isom^+(\half^{n+1})$. Then we extend our method to higher rank groups in order to prove an effective equidistribution theorem for expanding horospheres in quotients of $\SL_n(\R)/\SO_n(\R)$. This relates to numerous papers, for example \cite{EinsiedlerMargulisVenkatesh2009, KleinbockMargulis2012,MohammadiGolsefidy2014, Yang2016, AkaETAL2020}, and several others.  

  \begin{remark}
    It is worth noting that a related abstract spectral method was used by Str\"{o}mbergsson \cite{Strombergsson2013}, S\"{o}dergren \cite{Sodergren2012}, and Edwards \cite{Edwards2017, Edwards2021} to tackle the same types of problems. In fact in some contexts their method achieves stronger error terms. That is, because they can avoid smoothing, their $T$ dependence (the height of the horosphere) is improved, but they sacrifice in the dependence on the regularity of the test function. Since, in application, the test functions are normally approximations to indicator functions, this dependence can be very important. In addition, the method developed herein allows us to tackle somewhat more general problems: in the higher rank case we can expand in different directions with different rates. 

  \end{remark}

  \subsection{Plan of paper}
  In the remainder of the introduction, we present our main theorems for $T^1(\half^{n+1})$, $\SL_3(\R)$, and $\SL_n(\R)$ in that order. The rest of the paper is arranged as follows
  \begin{itemize}
  \item Section \ref{s:Preliminaries} presents some preliminaries on rank one hyperbolic geometry and spectral theory.
  \item  Section \ref{s:n=1} presents a simplified version the proof of Theorem \ref{thm:main asymptotic} in the special case when $n=1$ and the function is right $K$-invariant. This acts as a warm-up.
  \item Section \ref{s:rank 1} gives the full proof of Theorem \ref{thm:main asymptotic} for general $n$. That is, the rank one setting.
  \item Section \ref{s:prelimRank} presents the necessary preliminaries in the higher rank setting.
  \item Section \ref{s:SL3R} presents the proof of Theorem \ref{thm:SL3R}, the $\SL_3(\R)$ result.
  \item Section \ref{s:SLnR} then presents the proof of Theorem \ref{thm:SLnR}, the $\SL_n(\R)$ result.
  \end{itemize}
  Where possible we will avoid repeating details.

  \subsection{Main theorem -- rank $1$}
  
  Let $n \ge 1$ be fixed, and consider the $n+1$-dimensional hyperbolic upper half-space
  \begin{align*}
    \half^{n+1} := \{(x_1, \dots, x_n,y) \ : \ y >0 \}.
  \end{align*}
  Let $G:= \Isom^+(\half^{n+1})$, the group of orientation preserving isometries; which is isomorphic to the identity component of $\SO(n+1,1)$. When $n=1,2$, $G$ is isomorphic to $\PSL_2(\R)$ or $\PSL_2(\C)$ respectively. Fix a discrete, not co-compact, Zariski dense, geometrically finite subgroup $\Gamma < G$. Moreover assume the critical exponent $\delta > n/2$.

  Consider the right regular representation of $G$ acting on $L^2(\Gamma \bk G)$ (see Section \ref{s:Preliminaries} for definitions). We will also assume that the right regular representation of $G$ acting on the space $L^2(\Gamma \bk G)$ has no complementary series representation. We call such a  subgroup $\Gamma$ 'nice'. For $n=1$ and $n=2$ all of the groups we consider are nice. For $n > 2$ this is quite a restriction, but we impose it to avoid having to dismiss terms coming from these non-spherical complementary series representations. Recently, Edwards-Oh have treated the contribution of these non-spherical complementary series representations to a related problem in some detail \cite{EdwardsOh2021}, however it is simpler for us to restrict our results to nice subgroups.

  Fix a Cartan decomposition of $G = HAK$ where $K$ is a maximal compact subgroup, $A:=\{a_y \ : \ y\in \R_{>0} \}$ is a one parameter diagonalizable subgroup, and $H:= \{ g\in G \ : \ a_y g a_{1/y} \to e \text{ as } y \to 0 \}$ is the expanding horospherical subgroup for $a_{1/y}$. Further, let $M$ be the centralizer of $A$ in $K$. Given an $F \in C_c^\infty(\Gamma \bk G/M)$, our goal in this paper is to study the asymptotic behavior of the average
  \begin{align}
    \cM_T(F) : = \int_{h\in \Gamma_H \bk H} F(h a_{1/T}) \rd h,
  \end{align}
  in the limit as $T\to \infty$. Here $\Gamma_H:= \Gamma\cap H$ and $\rd h$ denotes the Haar measure on $H$. 

  To state our main theorem let
  \begin{align*}
    -\Delta := y^2 (\partial_{x_1x_1} + \dots + \partial_{x_nx_n} + \partial_{yy}) - (n-1) y \partial_y
  \end{align*}
  denote the hyperbolic Laplacian. Then, from work of Lax-Phillips \cite{LaxPhillips1982}, the $L^2(\Gamma \bk \half^{n+1})$ spectrum of the Laplacian consists of finitely many discrete eigenvalues $\lambda_0 < \lambda_1 \le \dots \le \lambda_k$ in the region $[0,(n/2)^2)$ and continuous spectrum above $n^2/4$. Write $\lambda_j = s_j(n-s_j)$ with $n/2 \le s_j\le \delta$. Then by work of Patterson \cite{Patterson1976}, and Sulliven \cite{Sullivan1984}, we have $s_0 = \delta$.

    Finally, let $\|F\|_{1,\infty} : = \| \partial_y F(x+iy) \|_{\infty}$ denote the first Sobolev norm of $F$, where $\|\cdot\|_{\infty}$ is the $L^\infty$ norm on $\Gamma \bk G$. Throughout we let $\|F\|_\Gamma = \|F\|_{L^2(\Gamma \bk G)}$. To ease notation, let $P(n):= n^2-3n+10$. Our main theorem, in rank one, is the following
    
  \begin{theorem}\label{thm:main asymptotic}
    Let $\Gamma<G$ be a 'nice', geometrically finite, Zariski dense subgroup with  critical exponent $\delta > n/2$. Let $H$ be a horospherical subgroup. Then for any $F \in C_c^\infty(\Gamma \bk G/M)$,  there exists a constant $c_0,$ such that
    \begin{align}
      \cM_T(F) = c_0 T^{\delta -n}(1 + O(T^{-\eta_{\text{cont}}}(\log T)^{\frac{4}{P(n)}} + T^{-\eta_{s_1}} )\|F\|_\Gamma^{4/P(n)} \|F\|_{1,\infty}^{\frac{P(n)-4}{P(n)}})
    \end{align}
    where
    \begin{align*}
      \eta_{\text{cont}} = \frac{4(\delta-n/2)}{P(n)}
      \qquad \text{ and } \qquad
      \eta_{s_1} = \frac{4(\delta - s_1)}{P(n)}.
    \end{align*}
    Further, if $F$ is right $K$-invariant, there exist constants $c_0, c_1 \dots , c_k$ such that
    \begin{align}\label{Mav K}
      \cM_T(F) = c_0 T^{\delta -n} + c_1 T^{s_1-n} + \dots + c_k T^{s_k-n} + O(T^{\delta/3-2n/3}(\log T)^{2/3}\|F\|_\Gamma^{2/3} \|F\|_{1,\infty}^{1/3}),
    \end{align}
    \eqref{Mav K} holds even if $\Gamma$ is not nice.
  \end{theorem}

  \begin{remark}
    This theorem improves on \cite[Theorem 1.7]{MohammadiOh2015} in that the error exponent $\eta$ is sharper. While \cite{MohammadiOh2015} do not state an explicit error term, a similar method is used in \cite{LeeOh2013} in the case $n=1,2$. For $n=2$ their $\eta_{s_1} = 2(\delta-s_1)/7$ while ours is $(\delta-s_1)/2$. Moreover Oh-Shah \cite{OhShah2013} show that the coefficient $c_0$ is proportional to $\BR(F)$ -- the Burger-Roblin mass of $F$.
    However, while \cite{Edwards2017} achieves a stronger error term in the $T$ aspect, Theorem \ref{thm:main asymptotic} is stronger in the Sobolev norm of $F$. Thus, for example when counting circles in an Apollonian circle packing, Edwards achieves an error term of size $T^{\delta -(\delta-s_1)/9}$ while applying Theorem \ref{thm:main asymptotic} achieves an error term of size $T^{\delta-2(\delta-s_1)/9}$ (although both results are worse than the more direct method of Kontorovich and the author \cite{KontorovichLutsko2022}). 
  \end{remark}
  \begin{remark}
    Note that for $K$-invariant functions we can extract all lower order terms coming from the discrete spectrum. This is a consequence of the fact that the subspace of $K$-fixed vectors in the $j^{th}$ eigenspace is one dimensional, while the full $j^{th}$ eigenspace is infinite dimensional, which prevents us from extracting some dependencies. 
  \end{remark}

  \begin{remark}
    One could possibly use the machinery of Edwards-Oh \cite{EdwardsOh2021} to remove the condition that $\Gamma$ is nice. We leave this extension for future work. 
  \end{remark}

  \subsection{Main theorem -- $\SL_3(\R)$}

  To keep matters simple, for higher rank, we first present our results on $\SL_{3}(\R)$. For this subsection, let $G:=\SL_3(\R)$. Consider the following subgroups
  \begin{itemize}
  \item $A:= \left\{ a(\vect{y}):= \begin{pmatrix} y_1y_2 & 0 & 0 \\ 0 & y_1 & 0 \\ 0  & 0 & 1 \end{pmatrix} \ : \ \vect{y}\in \R_{>0}^2 \right\}$ the two-parameter diagonal subgroup. 
  \item $H:= \left\{ h(\vect{x}):= \begin{pmatrix} 1 & x_2 & x_3 \\ 0 & 1 & x_1 \\ 0  & 0 & 1 \end{pmatrix} \ : \ \vect{x}\in \R^3 \right\}$ the upper-triangular subgroup.
  \item $K = \SO(3,\R)$ the maximal compact subgroup of $G$.
  \end{itemize}
  Then $\GL_3(\R)=HAKZ_3$ where $Z_3$ is a one parameter diagonal subgroup. Let $\cH:= G/ KZ_3\cong H A$ denote the generalized upper half plane (see \cite{Goldfeld2006}) or the space of lattices. The natural, invariant metric on $\cH$ is
  \begin{align*}
    \rd V = \frac{\rd x_1 \rd x_2 \rd x_3 \rd y_1 \rd y_2}{y_1^3 y_2^{3}}
  \end{align*}
  There are two Casimir operators in this case
  \begin{align*}
    -\Delta_1&:=y_1^2\partial_{y_1y_1} + y_2^2 \partial_{y_2y_2}-y_1y_2\partial_{y_1y_2} + y_1^2(x_2^2+y_2^2)\partial_{x_3x_3}+ y_1^2 \partial_{x_1x_1} + y_2^2\partial_{x_2x_2}+2y_1x_2 \partial_{x_1x_3},\\
      -\Delta_2 &:= -y_1^2y_2\partial_{y_1y_1y_2} + y_1y_2^2 \partial_{y_1y_2y_2} -y_2^2\partial_{y_2y_2}+y_1^2\partial_{y_1y_1}\\
      &\phantom{:=}-y_1^3y_2^2\partial_{x_3x_3y_1} + y_1 y_2^2 \partial_{x_2x_2y_1} -2y_1^2y_2x_2\partial_{x_1x_3y_2}+(-x_2+y_2)y_1^2y_2\partial_{x_3x_3y_2}\\
    &\phantom{:=}-y_1^2y_2\partial_{x_1x_1y_2}+2y_1^2y_2^2\partial_{x_1x_2x_3}+2y_1^2y_2x_2\partial_{x_2x_3x_3}  + 2y_1^2x_2\partial_{x_1x_3} +y_1^2(x_2^2+y_2^2)\partial_{x_3x_3}\\
    &\phantom{:=}  + y_1^2\partial_{x_1x_1}-y_2^2\partial_{x_2x_2}.
    \end{align*} 

  Let $\Gamma<\SL_3(\Z)$ be a co-finite subgroup. Consider the Hilbert space $\mathscr{H} = L^2(\Gamma \bk G)$, $G$ acts by the right regular representation on $\mathscr{H}$, which decomposes into irreducible subspaces (see Section \ref{s:prelimRank}). Both $\Delta_1$ and $\Delta_2$ act as scalars on each irreducible subspace. Thus, up to multiplicity we can parameterize the irreducible subspace by a pair $(\lambda_1, \lambda_2)$. If we consider the spectrum of $\Delta_1$, then this consists of finitely many eigenvalues
  \begin{align*}
    0=\lambda_1^{(0)} < \lambda_1^{(1)} < \dots < \lambda_1^{(k_1)} < 1
  \end{align*}
  of finite multiplicity. With continuous spectrum above $1$. 
  Since we work with co-finite subgroups, the base eigenvalue is $0$, and since the group satisfies property $T$, we know there is a spectral gap. In particular for $\Gamma = \SL_3(\Z)$ we have that $\lambda_1^{(1)}>1$ (see \cite{Miller2002}). Let $\vect{\lambda}^{(i)}$ for $i=0, \dots, k$ denote the discrete points in the joint spectrum of $\lambda_1$ and $\lambda_2$. Any point in the discrete spectrum has finite multiplicity.

  Rather than work with $\vect{\lambda}^{(i)}$ it is more practical to work with vectors $\vect{\nu}=(\nu_1,\nu_2) \in [0,1)^2$. Let $(\nu_1,\nu_2)$ solve the equations $\lambda_1 =-3(\nu_1^2 +\nu_2^2-\nu_1-\nu_2+\nu_1\nu_2 $ coming from $\Delta_1$ and $\lambda_2 = (\nu_2-\nu_1)(2\nu_1^2+\nu_2(-3+2\nu_2)+\nu_1(-3+5\nu_2)$, coming from $\Delta_2$. These equations come from eigenvalue equations for the $I$-function $I(z) = y_1^{\nu_1+2\nu_2}y_2^{2\nu_1+\nu_2}$ and there are $6$ solutions to the pair of equations, call them $(\nu_{i,1},\nu_{i,2})$ for $i=1, \dots, 6$. The continuous spectrum corresponds to $\nu_1= 1/3+it_1$ and $\nu_2 = 1/3+ it_2$. For convenience we also define $s=\nu_1+2\nu_2$ and $r=\nu_2+2\nu_1$ (the exponents of $y_1$ and $y_2$ respectively, in the $I$ function).



  For $F\in C_c^\infty(\Gamma \bk \cH)$, our aim is to study the asymptotic behavior of
  \begin{align*}
    \cM_{\vect{T}}(F) : = \int_{\Gamma_H \bk H} F(h a(1/T_1,1/T_2)) \rd h.
  \end{align*}
  where $\vect{T}=(T_1,T_2)$. Now, for any $\vect{\lambda} = (\lambda_1, \lambda_2)$ define the following quantity
  \begin{align}\label{m def}
    m_{\vect{T}}(\vect{\lambda}): = \sum_{i=1}^6 c_i T_1^{s_i}T_2^{r_i}, 
  \end{align}
  for some constants $c_i$ which could be made explicit.
  We need this cumbersome quantity since the term governing the error term will depend on the (a priori unknown) relationship between $s$, $r$, $T_1$ and $T_2$.
  
  \begin{theorem}\label{thm:SL3R}
    Let $F\in C_c^\infty(\Gamma \bk \cH)$, then there exist constants $c_0,c_1, \dots, c_k<\infty$, depending only on $\Gamma$ and $F$, such that
    \begin{align}
      \cM_{\vect{T}}(F) = &c_0m_{\vect{T}}(0)  + c_1 m_{\vect{T}}(\vect{\lambda}^{(1)}) + \dots + c_k m_{\vect{T}}(\vect{\lambda}^{(k)})\\
      &\phantom{+++++++++}+ O(T_1^{-1/2}T_2^{-1/2}(\|F\|_{1,\infty} \|F\|_\Gamma\log T_1 \log T_2)^{1/2}),\notag
    \end{align}
    in the limit as $T_1T_2 \to \infty$. In particular, if $T_1,T_2 \to \infty$ the main term is $C \int_{\Gamma \bk \cH} F(z) \rd z$.
  \end{theorem}
  In the higher rank we write $\|F\|_{1,\infty} : = \| \nabla_{\vect{y}} F(\vect{x},\vect{y})\|_\infty$.

  \begin{remark}
    A similar proof method was used by Edwards \cite{Edwards2021} to achieve an excellent error term in the more general Lie group setting. Theorem \ref{thm:SL3R} is still worthwhile proving, given the simplicity and directness of our method here, and since the $F$ dependence is extremely good. Thus, while Edwards achieves a better $T$ dependence, he sacrifices in the $F$ dependence. 
  \end{remark}

  \subsection{Main theorem -- $\SL_n(\R)$}

  Turning now to the group $G:=\SL_n(\R)$ for $n \ge 2$. Again, we employ an Iwasawa decomposition. Consider the following subgroups
  \begin{itemize}
  \item $A:= \left\{ a(\vect{y}):=\begin{pmatrix}
    y_1\cdots y_{n-1} & & & & \\
    & y_1 \cdots y_{n-2} & & & \\
    & & \ddots & & \\
    & & & y_1 & \\
    & & & & 1
    \end{pmatrix}
    \ : \ \vect{y}\in \R_{>0}^{n-1} \right\}$ the $n-1$-parameter diagonal subgroup. 
  \item $H:= \left\{ h(\vect{x}):= \begin{pmatrix}
    1 & x_{1,2} & x_{1,3} & \dots &  & x_{1,n} \\
     & 1 & x_{2,3} & \dots &  & x_{2,n} \\
     &  & \ddots &  &    & \vdots \\
     &  &  & \dots & 1 & x_{n-1,n} \\
     &  &  &  &  & 1 
  \end{pmatrix} \ : \ x_i \in \R \right\}$
    the upper-triangular subgroup of dimension $n(n-1)/2$.
  \item $K = \SO_n(\R)$ the maximal compact subgroup of $G$.
  \end{itemize}
  Once again we can write $\GL_n(\R)=HAKZ_d$ with $Z_d$ a one-parameter subgroup. Let $\cH:=HA$ denote the generalized upper half-space. The invariant metric on $\cH$ is
  \begin{align*}
    \rd V = \prod_{1\le i < j\le n}\rd x_{i,j}\prod_{k=1}^{n-1} y_k^{-k(n-k)-1}\rd y_k.
  \end{align*}
  Let $\Delta_1, \dots, \Delta_{n-1}$ denote the Casimir operators.

  Let $\Gamma < \SL_n(\Z)$ be a cofinite subgroup with $\Gamma_H\bk G$ is closed. Again, the Casimir operators act like scalars on the space of irreducibles. We denote the exceptional values (that is the values outside the tempered spectrum)
  \begin{align*}
    0 = \vect{\lambda}^{(0)} < \vect{\lambda}^{(1)} \le \dots \le \vect{\lambda}^{(k)},
  \end{align*}
  where $\lambda_i^{(j)}$ is the value of $\Delta_i$ acting on that irreducible subspace. 

  Rather than work with the spectral parameters $\vect{\lambda}$ it is more convenient to  work with the corresponding $\vect{\nu}$ coming from the $I$-functions. That is, set
  \begin{align*}
    I_{\vect{\nu}}(\vect{y}) : = \prod_{i=1}^{n-1}\prod_{j=1}^{n-1} y_i^{b_{ij}\nu_j}
  \end{align*}
  where
  \begin{align*}
    b_{ij} := \begin{cases}
      ij, &\mbox{ if } i+j\le n,\\
      (n-i)(n-j),& \mbox{ if }i+j\ge n.
    \end{cases}
  \end{align*}
  Furthermore, let $s_i:= \sum_{j=1}^{n-1} b_{ij}\nu_j$ denote the exponent of $y_i$ in the $I$ function. Then the $\vect{\nu}$ solve the system of equations
  \begin{align*}
    \Delta_k I_{\vect{\nu}}(\vect{y}) = \lambda_k \Delta_k I_{\vect{\nu}}(\vect{y}), \qquad \mbox{ for all } k = 1, \dots, n-1
  \end{align*}
  for a given point in the spectrum there are $L= n!$ such solutions, denoted $\vect{\nu}_1, \dots, \vect{\nu}_L$. For the discrete points in the exceptional spectrum (which are expected not to exist) $\vect{\lambda}^{(i)}$ we denote the associated $\vect{\nu}_j$ by $\vect{\nu}^{(i)}_j$. The exceptional spectrum lies in $[0,1/n]^{n-1}$ \cite{Goldfeld2006} and the continuous spectrum corresponds to points $\vect{\nu}$ with $\nu_j= 1/n+it_j$ for each $j =1 , \dots, n-1$. For convenience we define the following quantity
  \begin{align*}
    m_{\vect{T}}(\vect{\lambda}) = \sum_{j=1}^{L} I_{\vect{\nu}_j}(\vect{y}).
  \end{align*}

  Our goal is to understand the asymptotic behavior of the following horospherical average, let $F \in C^\infty_c(\Gamma \bk \cH)$ and let $\vect{T} \in \R^{n-1}_{>0}$
  \begin{align*}
    \cM_{\vect{T}}(F):= \int_{\Gamma_H\bk H}F(ha(1/\vect{T}))\rd h.
  \end{align*}
  For that purpose, let $I_{cont}(\vect{T}): = \prod_{i=1}^{n-1}T_i^{(b_{i1}+\dots + b_{i(n-1)})/n}$. The equidistribution theorem for $\SL_n(\R)/ \SO(n,\R)$ is the following

  \begin{theorem}\label{thm:SLnR}
    Suppose $\Gamma$ is as above. Let $F\in C_c^\infty(\Gamma \bk \cH)$, then there exist constants $c_0,c_1, \dots, c_k<\infty$, depending only on $\Gamma$ and $F$, such that
    \begin{align}
      \cM_{\vect{T}}(F) &= c_0 m_{\vect{T}}(0)  + c_1 m_{\vect{T}}(\vect{\lambda}^{(1)}) + \dots + c_k m_{\vect{T}}(\vect{\lambda}^{(k)})\\
      &\phantom{+++++}+ O \left( \left(I_{cont}(\vect{T}) \log T_1 \cdots \log T_{n-1}\right)^{\frac{2}{n+1}} \|F\|_\Gamma^{\frac{2}{n+1}} \|F\|_{1,\infty}^{\frac{n-1}{n+1}} \right),\notag
    \end{align}
    in the limit as $\prod_{i=1}^{n-1} T_i \to \infty$. In particular, if all $T_i \to \infty$ the main term is $C \int_{\Gamma \bk \cH} F(z) \rd z$ for some $C>0$.
  \end{theorem}

  \section{Preliminaries and Notation -- Rank One}
  \label{s:Preliminaries}

  Given an $n \ge 1$ we work in the group $G:=\Isom(\half^{n+1})$. Throughout, let $\<\cdot, \cdot\>_\Gamma$ denote the $L^2(\Gamma \bk G)$ inner product with respect to the Haar measure $\rd g$. Let $\mathfrak{g}:= \Lie(G)$, note that $d:= \dim(\mathfrak{g}) = (n+1)(n+2)/2$. We have the following subgroups
  \begin{itemize}
  \item $A := \{ a_y \ : \ y>0 \}$ - a one dimensional diagonalizable subgroup.
  \item $H := \{ g\in G \ : \ a_y g a_{1/y} \to e \text{ as } y \to 0 \}$ - the expanding horospherical subgroup associated to $A$ -- $\dim(H) = n$.
  \item $K$ - the maximal compact subgroup -- $\dim(K) = (n+1)n/2 $.
  \item $M$ - the centralizer of $A$ in $K$ -- $\dim(M) = n-1 $.
  \end{itemize}
  Further, let $\Kbar := K/M$. Note that $\half^{n+1} = G/K$ and $T^1(\half^{n+1}) = G/M$. Throughout, we use $\Gamma_X$ to denote $\Gamma \cap X$.

\textbf{Lie algebras and the Casimir Operator:} In general, the Casimir operators generate the center of the universal enveloping algebra. The elements of the Lie algebra act like first order differential operators, and the Casimir operator acts as a second order differential operator on smooth functions on $G$. 
When restricted to right $K$-invariant smooth functions on $G$, 
the Casimir operator $\cC$ agrees with the hyperbolic Laplacian $\Delta$.

\textbf{Decomposition of $L^2(\G \bk G)$ into irreducibles:} The group $G$ acts  by the right-regular representation on the Hilbert space $\mathscr{H} := L^2(\Gamma\bk G)$ of square-integrable $\Gamma$-automorphic functions. 
The space $\mathscr{H}$ splits into components as follows:
\be\label{eq:cHdecomp}
\mathscr{H}=\mathscr{H}_0\oplus \mathscr{H}_1\oplus\cdots\oplus\mathscr{H}_k\oplus \mathscr{H}^{tempered}.
\ee
Each of the subspaces $\mathscr{H}_j$ denotes the $G$-span of the eigenfunction with eigenvalue $\lambda_j$. $\mathscr{H}^{tempered}$ denotes the tempered spectrum. Since we work with nice subgroups, there is no nonspherical complimentary series. For $n \ge 2$ the space $\mathscr{H}_j$ is infinite dimensional. However the subspace of $K$-fixed vectors in $\mathscr{H}_j$ is one dimensional.

\subsection{Abstract Spectral Theorem:}

The abstract spectral theorem is a remarkable theorem coming from abstract operator theory (see for example \cite[Ch. 13]{Rudin1973}). Since it represents the crucial input into our method we state it in full generality here. Let $L$ be a self-adjoint, positive semidefinite operator on the Hilbert space $\mathscr{H}$. In our applications $\mathscr{H}$ will be $L^2(\Gamma \bk G)$ and $L$ will be the Casimir operator (or Laplacian).

\begin{theorem}[Abstract Spectral Theorem]\label{thm:AST}
    There exists a spectral measure $\wt{\mu}$ on $\R$ and a unitary spectral operator $\wh{\phantom{\cdot\cdot}}: \mathscr{H} \to L^2([0,\infty),d \wt{\mu})$ such that:
    \begin{enumerate}[label = \roman*)]
    \item Abstract Parseval's Identity: for $\phi_1,\phi_2 \in \mathscr{H}$
        \begin{align}\label{API}
            \<\phi_1,\phi_2\>_{\mathscr{H}}=\<\wh{\phi_1},\wh{\phi_2}\>_{L^2([0,\infty), d\wt{\mu})};
        \end{align}
    \item The spectral operator is diagonal with respect to $L$: for $\phi\in \mathscr{H}$ and $\lambda \ge 0 $,
        \begin{align}
            \wh{L\phi}(\lambda) = \lambda \wh{\phi}(\lambda);
        \end{align}
    \end{enumerate}
\end{theorem}

If $\lambda$ is in the point specturm of $L$ with associated
eigenspace $\mathscr{H}_\gl$, then for any $\psi_1,\psi_2\in \mathscr{H}$ one has
        \begin{align}\label{proj}
            \wh{\psi_1}(\lambda)\wh{\bar\psi_2}(\lambda) 
            =
            \<\operatorname{Proj}_{\mathscr{H}_\gl}\psi_1,
            \operatorname{Proj}_{\mathscr{H}_\gl}\psi_2\>,
        \end{align}
where $\operatorname{Proj}$ refers to the projection to the subspace $\mathscr{H}_\gl$. In the special case that $\mathscr{H}_\gl$ is one-dimensional and spanned by the normalized eigenfunction $\phi_\gl,$  we have that
        \begin{align}
            \wh{\psi_1}(\lambda)\wh{\bar\psi_2}(\lambda) 
            =
            \<\psi_1,\phi_\gl\>
            \<\phi_\gl,\psi_2\>.
        \end{align}

    \section{Expanding Horospheres in $\half^{1+1}$}
    \label{s:n=1}

  Let $\Gamma<\SL_2(\R)$ and assume $\delta>1/2$. Suppose the discrete spectrum of $\Gamma$ has $k$ many values above the base, that is, there exist eigenvalues $\delta(1-\delta) = \lambda_0 < \lambda_1 < \dots < \lambda_k < \frac{1}{4}$. Further, write $\lambda_k= s_k(1-s_k)$.  The purpose of this section is to prove the following equidistribution theorem, which is the $n=1$ case of Theorem \ref{thm:main asymptotic} for $K$-invariant functions.  

  \begin{theorem}[Horocyclic Equidistribution]\label{thm:fin vol}
     Suppose $\Gamma_H \bk H$ is closed. Let $F \in C_c^\infty(\Gamma \bk \half)$, then the horocyclic average satisfies
    \begin{align}\label{fin vol}
      \cM_T(F) = c_0 T^{\delta-1} + c_1 T^{s_1-1} + \dots + c_k T^{s_k-1} + O(T^{\delta/3-2/3} (\log T)^{2/3})
    \end{align}
    where $c_i$ depend only on the group $\Gamma$ and $F$, for all $i =0, 1 , \dots, k$ .

  \end{theorem}

  \begin{remark}
    For the modular surface we recover a weaker version of Sarnak and Zagier's results \cite[Theorem 1]{Sarnak1981} \cite{Zagier1981} which do not rely on any smoothing procedure, but exploit the exact form of the modular surface.
  \end{remark}

  If $\Gamma_H \bk H$ is closed then we can assume the cusp at $\infty$ has width $[-X,X]$. If not, then since $F$ has compact support, we can assume the support is contained in $x \in [-X,X]$.  Then the average becomes
  \begin{align*}
    \int_{-X}^X F(x+i/T) \rd x  
  \end{align*}
  where $\cF$ is a fundamental domain for $\Gamma$. 

  Let
  \begin{align*}
    \psi_{T,\vep} (y) : = \frac{1}{2\vep} \one(y \in [1/T-\vep/T^2,1/T+\vep/T^2])
  \end{align*}
  be the $L^1$ normalized indicator function.
  We write $\psi_{\vep,T}(z) = \psi_{\vep,T}(y)$ for $z=x+iy$. Further, let $\Psi_{\vep,T}:\Gamma \bk \half \to \R$ denote the automorphization of $\psi_{\vep,T}$, that is
  \begin{align*}
    \Psi_{\vep,T}(x+iy): = \sum_{\gamma \in \Gamma_\infty \bk \Gamma} \psi_{\vep,T}(\gamma z ),
  \end{align*}
  where $\Gamma_\infty$ could be trivial.
  Now define the $\vep$-thickened average to be
  \begin{align*}
    \cM_{\vep}(T,F) = \cM_\vep (T) := \int_{\cF} F(x+iy ) \Psi_{\vep,T}(x+iy ) \frac{\rd x \rd y}{y^2}.
  \end{align*}

  \subsection{Differential Equation}
  Now by unfolding we arrive at
  \begin{align*}
    \cM_\vep (T) &= \< F ,  \Psi_{\vep,T} \>_{\Gamma \bk \half}\\
    &= \int_{\cF} F(z) \sum_{\gamma \in  \Gamma_\infty \bk \Gamma} \psi_{\vep, T} (\gamma z)  \rd z\\
    &= \int_{\Gamma_\infty \bk \half} F(z) \psi_{\vep,T}(z) \rd z\\
    &= \int_{\R} \psi_{\vep,T}(y) \int_{-X}^X F(x+iy)  \frac{\rd x \rd y}{y^2}.
  \end{align*}
  Let $f(y): =  \int_{-X}^X F(x+iy) \rd x $, fix $\lambda=s(1-s) >0$ and let $ g(y) :=  y^2 \partial_{yy} f(y) - \lambda f(y)$. Then $f$ satisfies the differential equation
  \begin{align*}
    y^2\partial_{yy} f(y) - \lambda f(y) = g(y) = \int_{-X}^{X} (\Delta-\lambda)F \rd x.
  \end{align*}
  Hence $f$ satisfies \cite[Lemma B.1]{Kontorovich2009}. That is, if $\lambda \neq 1/4$ there exist constants $A$, $B$ and functions $u$, $v$ such that
  \begin{align}
    f(y) = A y^s + B y^{1-s} + u(y) y^s + v(y) y^{1-s}.
  \end{align}
  When $\lambda = 1/4$, we may similarly conclude that
  \begin{align}
    f(y) = A y^{1/2} + B y^{1/2}\log y + u(y) y^{1/2} + v(y) y^{1/2}\log y.
  \end{align}
  For simplicity we henceforth assume that $\lambda \neq 1/4$ (if $\lambda=1/4$ the same argument can be applied -- see \cite{Kontorovich2009}). In this case we have 
  \begin{align*}
    u(y) = (1-2s)^{-1} \int_{\frac{1}{T} - \frac{\vep}{T^2}}^y w^{-1-s} g(w) \rd w,
    \qquad
    v(y) = (2s-1)^{-1} \int_{\frac{1}{T} - \frac{\vep}{T^2}}^y w^{s-2} g(w) \rd w.
  \end{align*}

  Now we integrate $f(y)$ against $\frac{\psi_{\vep,T}}{y^2}$. First, for the main term we can apply a power series expansion around $1/T$ to conclude
  \begin{align*}
    \int_{0}^{\infty} \psi_{\vep,T}(y) (A y^{s} + B y^{1-s}) \frac{\rd y}{y^2} = A \alpha(T) + B \beta(T),
  \end{align*}
  where
  \begin{gather}
    \begin{gathered} \label{alph beta bounds}
    \alpha(T) : = \int_0^\infty\psi_{\vep,T}(y) y^s \frac{\rd y}{y^2} = \vep^{-1}\left((\frac{1}{T}-\frac{\vep}{T^2})^{s-1} - (\frac{1}{T}+\frac{\vep}{T^2})^{s-1}\right)   =  T^{-s} + O(\vep T^{-(s+1)}),\\
    \beta(T) :  = \int_0^\infty\psi_{\vep,T}(y) y^{1-s} \frac{\rd y}{y^2} = \vep^{-1}\left((\frac{1}{T}-\frac{\vep}{T^2})^{-s} - (\frac{1}{T}+\frac{\vep}{T^2})^{-s}\right)   =  T^{s-1} + O(\vep T^{s-2}).
    \end{gathered}
  \end{gather}

  As for the contribution from $u$, let
  \begin{align*}
    I = \int_{0}^\infty \psi_{\vep,T}(y) y^s u(y) \frac{\rd y}{y^2},
  \end{align*}
  then, by  integrating by parts we have (with $y_{\min}= 1/T - \vep/T^2$ and $y_{\max} = 1/T + \vep/T^2$)
  \begin{align*}
    I = c \int_{y_{\min}}^{y_{\max}}  g(y) \frac{\rd y}{y^2} + cy_{\max}^{s-1} \int_{y_{\min}}^{y_{\max}} y^{-1-s} g(y) \rd y.
  \end{align*}
  Now apply Cauchy-Schwarz, yielding
  \begin{align*}
    I &\ll_{\lambda, T,\vep} \int_{y_{\min}}^{y_{\max}}\int_{-X}^X   y^{-s} \left(\frac{(\Delta - \lambda) F(z)}{y}\right) \rd x\rd y,\\
    &\ll_{\lambda,T,\vep} \left(\int_{y_{\min}}^{y_{\max}}\int_{-X}^X \abs{ y^{-s}}^2\rd x\rd y\right)^{1/2}\left(\int_{y_{\min}}^{y_{\max}}\int_{-X}^X  \abs{\frac{(\Delta - \lambda) F(z)}{y}}^2 \rd x\rd y\right)^{1/2},\\
    &\ll_{\lambda,T,\vep} \|(\Delta - \lambda)F \|_{\Gamma}.
  \end{align*}
  The same bound can be derived for the term coming from $v(y)$. Hence we conclude that for any $F$, we have
  \begin{align} \label{M exp}
    \cM_\vep (T,F) = A \alpha(T) + B \beta(T) +O(\|(\Delta -\lambda)F\|).
  \end{align}

  \subsection{Inserting the Laplacian}

  Now suppose $F$ \emph{were} an eigenfunction of $\Delta$ with eigenvalue $\lambda$. Then the final term in \eqref{M exp} vanishes.  Plugging in two values $T=1$ and $T=b$, we can solve for $A$ and $B$. A short calculation shows (for $\lambda \neq 1/4$)
  \begin{align}\label{M eigen}
    \cM_{\vep}(T,F) =  K_T(s)\cM_{\vep}(1,F) + L_T(s)\cM_{\vep}(b,F),
  \end{align}
  where
  \begin{align}\label{KL def}
    K_T(s) := \frac{\beta(b)\alpha(T) - \alpha(b)\beta(T)}{\alpha(1)\beta(b)-\alpha(b)\beta(1)},
    \qquad
    L_T(s) := \frac{\alpha(1)\beta(T) - \beta(1)\alpha(T)}{\alpha(1)\beta(b)-\alpha(b)\beta(1)}.
  \end{align}
  Note that since $s \in (1/2,1]$ we have $K_T(s),L_T(s) \asymp T^{s-1}$. Moreover, when $\lambda \ge 1/4$ then we have $K_T(1/2),L_T(1/2) \ll T^{-1/2}\log T$.

  We abuse notation and write $K_T(\lambda)=K_T(s)$ where $\lambda = s(1-s)$ and likewise for $L_T$. Furthermore, just as one can define the exponential of a matrix, we let $K_T(\Delta)$ and $L_T(\Delta)$ be the analogous operators for $K_T(s)$ and $L_T(s)$ respectively. The following theorem states that \eqref{M eigen} holds even if $F$ is not an eigenfunction of the Laplacian.

 \begin{theorem}[Main Identity]\label{thm:main ident finite}
   For fixed $T \ge 1$, there exists a number, $b$, such that
   \begin{align}\label{Psi main}
     \Psi_{\vep,T} = K_T(\Delta)\Psi_{\vep,1} + L_T(\Delta)\Psi_{\vep,b}
   \end{align}
   holds almost everywhere. Moreover $K_T $ and $L_T$ satisfy the bounds
   \begin{align}\label{KL bounds}
     K_T(s),L_T(s) \ll
     \begin{cases}
       T^{s-1} & \mbox{ if } s \in (1/2,1]\\
       T^{-1/2}\log T & \mbox{ if } s=1/2+it.
     \end{cases}
   \end{align}
   
 \end{theorem}

 \begin{proof}
   The proof follows an almost verbatim application of \cite[Proof of Theorem 3.2]{Kontorovich2009}.
 \end{proof}

 \subsection{Proof of Thickened Equidistribution}

 To simplify notation, assume there are no eigenvalues other than the base eigenvalue $\lambda_0$. Note that by applying the abstract Parseval's identity
 \begin{align*}
   \cM_{\vep}(T) &= \< \Psi_{\vep,T}, F\>_\Gamma\\
   &= \< \wh{\Psi_{\vep,T}}, \wh{F} \>_{\Spec(\Gamma)}\\
   &= \wh{F}(\lambda_0) \wh{\Psi_{\vep,T}}(\lambda_0) + \int_{\Spec(\Gamma)\setminus\{\lambda_0\}} \wh{F}(\lambda) \wh{\Psi_{\vep,T}}(\lambda) \rd\wt{\mu}(\lambda).
 \end{align*}
 Using Patterson-Sullivan theory we know what the base eigenfunction is, and thus, since $F$ is compactly supported, and thus integrable with respect to the Burger-Roblin measure we conclude that 
 \begin{align*}
   \wh{F}(\lambda_0) : = \< F, \phi_0\>
 \end{align*}
 is some finite, explicit constant. Furthermore, using Theorem \ref{thm:main ident finite} and \eqref{alph beta bounds} we have
 \begin{align*}
   \wh{\Psi_{\vep,T}}(\lambda_0) = T^{\delta-1}(c_1 \<\Psi_{\vep,1},\phi_0\>+c_2 \<\Psi_{\vep,b},\phi_0\>) + O(\vep T^{-\delta}).
 \end{align*}
 Hence, since $\Psi_{\vep,1}$ can be unfolded, and $\psi_{\vep,1}$ has unit mass, we can use the mean value theorem (as done in \cite[(4.16)]{Kontorovich2009}) to conclude
 \begin{align*}
   \<\Psi_{\vep,1},\phi_0\> =  C(1+O(\vep)).
 \end{align*}
 From which it follows that
 \begin{align*}
   \wh{F}(\lambda_0) \wh{\Psi_{\vep,T}}(\lambda_0) = C_{\Gamma,F} T^{\delta-1} (1+ O(\vep)).
 \end{align*}
 The lower order contribution from any exceptional eigenvalues can be similarly estimated to give $c_i T^{s_i-1} (1+ O(\vep))$.

 As for the error term, we may apply Theorem \ref{thm:main ident finite}, \emph{ii)} from the abstract spectral theorem (Theorem \ref{thm:AST}) and the estimate \eqref{KL bounds}, yielding
 \begin{align*}
   \int_{\Spec(\Gamma)\setminus\{\lambda_0\}} \wh{F}(\lambda) \wh{\Psi_{\vep,T}}(\lambda) \rd\wt{\mu}(\lambda)
   =
   \int_{\Spec(\Gamma)\setminus\{\lambda_0\}} \wh{F}(\lambda) (\wh{K_T(\Delta)\Psi_{\vep,1}}(\lambda) + \wh{L_T(\Delta)\Psi_{\vep,b}}(\lambda) \rd\wt{\mu}(\lambda)\\
   =
   \int_{\Spec(\Gamma)\setminus\{\lambda_0\}} \wh{F}(\lambda)( K_T(\lambda)\wh{\Psi_{\vep,1}}(\lambda) + L_T(\lambda)\wh{\Psi_{\vep,b}}(\lambda) \rd\wt{\mu}(\lambda)\\
   \ll
   T^{- 1/2} \log T \int_{\Spec(\Gamma)\setminus\{\lambda_0\}} \wh{F}(\lambda)(\wh{\Psi_{\vep,1}}(\lambda) +\wh{\Psi_{\vep,b}}(\lambda)) \rd\wt{\mu}(\lambda).
  \end{align*}
  To conclude we apply the abstract Parseval's identity and Cauchy-Schwarz giving
  \begin{align*}
    \int_{\Spec(\Gamma)\setminus\{\lambda_0\}} \wh{F}(\lambda) \wh{\Psi_{\vep,T}}(\lambda) \rd\wt{\mu}(\lambda)
    &\ll T^{- 1/2} \log T \|F\|_{\Gamma} \| \Psi_{\vep,1}\|_\Gamma.
  \end{align*}
  This is the source of the $\|F\|_{\Gamma}$ in the error term. As for the second factor we have
  \begin{align*}
    \| \Psi_{\vep,1}\|_\Gamma \ll \frac{1}{\sqrt{\vep}}.
  \end{align*}
  Leading to the following 'thickened' version of the equidistribution result
  \begin{align}\label{thick}
    \cM_{\vep}(T,F) = c_0 T^{\delta-1} + c_1 T^{s-1} + \dots c_k T^{s_k-1} + O(\vep^{-1/2}T^{-1/2} \log T),
  \end{align}
  where $c_i$ depend on the group, $\vep$, and on $F$, moreover, for each $i$ we can write $c_i = C_i(1+O(\vep))$, where $C_i$ is independent of $\vep$. 

  \subsection{Proof of Theorem \ref{thm:fin vol}}

    By the unit mass and positivity of $\Psi$,  we have that, after unfolding
  \begin{align*}
    \abs{\cM(T) - \cM_\vep(T)}  = \abs{\int_{1/T-\vep/T^2}^{1/T+\vep/T^2} \int_{-X}^X (F(x+i/T) - F(z)) \psi_{\vep,T}(z) \frac{\rd x \rd y}{y^2}}.
  \end{align*}
  Now apply Cauchy-Schwarz
  \begin{align*}
    &\ll \left(\int_{1/T-\vep/T^2}^{1/T+\vep/T^2} \int_{-X}^X (F(x+i/T) - F(z))^2 \rd z\right)^{1/2} \left(\int_{1/T-\vep/T^2}^{1/T+\vep/T^2} \int_{-X}^X \psi_{\vep,T}(z)^2 \rd z \right)^{1/2},\\
    &\ll \frac{1}{\sqrt{\vep}} \left(\int_{1/T-\vep/T^2}^{1/T+\vep/T^2} \int_{-X}^X (F(x+i/T) - F(z))^2 \rd z\right)^{1/2}.
  \end{align*}
  Finally apply the mean value theorem and the fact that $F$ is assumed to be compactly supported to conclude
  \begin{align}\label{M Meps}
    \begin{aligned}
      \abs{\cM(T) - \cM_\vep(T)}  &\ll \frac{\|F\|_{1,\infty}}{\sqrt{\vep}} \left(\vep^2 T^{2(\delta-1)} \int_{1/T-\vep/T^2}^{1/T+\vep/T^2} \int_{-X}^X 1 \rd z\right)^{1/2}\\
      &\ll \vep T^{\delta-1}\|F\|_{1,\infty},
      \end{aligned}
  \end{align}
  the factor of $T^{\delta-1}$ comes from the fact that $F$ is compactly supported in $\Gamma \bk \half$. Namely, since $F$ has compact support, the only fundamental domains which intersect a small neighbor hood around the horosphere of height $1/T$ and thus contribute to the above integral are those with maximum height near $1/T$. There are approximately $T^{\delta}$ such domains. Each such domain contributes $T^{-1}$ to the total integral. Thus we arrive at $T^{\delta-1}$ (this is a standard loss of mass argument).

  Finally to prove Theorem \ref{thm:fin vol} we put \eqref{thick} and \eqref{M Meps} together, giving
  \begin{align}
    \cM(T,F) = &c_0 \mu(F) T^{\delta-1} + c_1 T^{s-1} + \dots c_k T^{s-k} + O(\vep^{-1/2}T^{-1/2} \log T \|F\|_\Gamma)\\
    &\phantom{+++++++++++++++++++++}+ O(\vep T^{\delta-1}\|F\|_{1,\infty})\notag.
  \end{align}
  Now choose $\vep = T^{1/2-2\delta/3}(\log T)^{2/3} \|F\|_\Gamma^{2/3} \|F\|_{1,\infty}^{-2/3}$ to conclude \eqref{fin vol}.

  \section{Expanding Horospheres -- Rank 1}
  \label{s:rank 1}
  
  In this section we present the proof of Theorem \ref{thm:main asymptotic} for general $n$.
  To apply our method, we first thicken in the $y$-direction, then we additionally mollify in the $\Kbar$-direction. 
  To that end, fix a value $\vep>0$ to be determined later (depending on $T$). As we did when $n=1$ let $\psi_{\vep,T}(a_y)$ denote the scaled indicator function
  \begin{align*}
    \psi_{\vep,T}(a_y) : = \frac{1}{2\vep} \one\left( y\in \left[ \frac{1}{T} - \frac{\vep}{T^{n+1}}, \frac{1}{T} - \frac{\vep}{T^{n+1}}\right]\right).
  \end{align*}
  Thus $\psi_{\vep,T}$ is normalized against $\frac{\rd y}{y^{n+1}}$. Let $\psi_{\vep,T}(h a k) = \psi_{\vep,T}(a)$ and
  let
  \begin{align*}
    \Psi_{\vep,T}(g) : = \sum_{\gamma \in \Gamma_H \bk \Gamma} \psi_{\vep,T}(\gamma g)
  \end{align*}
  denote the automorphization of $\psi_{\vep,T}$.
  Furthermore, let $\psi_{\vep,K}(k)$ be a unit mass, even, bump function supported in a ball of radius $\vep$ around $k=e$.
  Given $F \in L^2(\Gamma \bk G/M)$ let
  \begin{align*}
    \wt{F} : = \int_{\Kbar} F(g k^{-1}) \psi_{\vep,\Kbar}(k) \ \rd k.
  \end{align*} 
  Now write
  \begin{align*}
    \cM_{\vep}(T) &: = \< \Psi_{\vep,T}, \wt{F} \>_\Gamma\\ 
    &=
    \int_{\Gamma \bk G} \Psi_{\vep,T}(g)\wt{F}(g) \ \rd g.
  \end{align*}

  \subsection{Unfolding to derive the differential equation}
  
  Let $\cX$ denote the range of $\vect{x}$ in the support of $\wt{F}$. Note that $\wt{F}$ has compact support since $F$ does, thus a particular coordinate of $\cX$ either represents the periodic length of the fundamental domain in a particular direction, or some finite interval inside of a non-compact direction. 
  Now unfold the inner product:
  \begin{align*}
    \cM_{\vep}(T) 
    &=    \int_{\Gamma \bk G} \wt{F}(g)\Psi_{\vep,T}(g)  \rd g\\
    &=    \int_{\Gamma \bk G} \wt F(g)\sum_{\Gamma_H\bk \Gamma} \psi_{\vep,T}(\gamma g)  \rd g\\
    &=    \int_{\Gamma_H \bk G} \wt F(g)  \psi_{\vep,T}( g)  \rd g\\
    &=    \int_0^\infty \psi_{\vep,T}(y) \int_{\cX} \int_K \wt F(h(\vect{x}) a_y k ) \    \   \frac{\rd k   \rd \vect{x}\rd y}{y^{n+1}}.
  \end{align*}
  We treat $\wt{F}$ as a general $L^2(\Gamma \bk G/M)$-function, and let
  \begin{align*}
    f(y): = \int_{\cX} \int_K \wt F(h(\vect{x}) a_y k ) \    \   \rd k   \rd \vect{x}.
  \end{align*}
  then $f(y)$ satisfies
  \begin{align}\label{diff eq}
    (y^2\partial_{yy} - (n-1) y \partial_y)f(y) = \int_{\cX} \int_K (\cC - \lambda)\wt F(h(\vect{x}) a_y k ) \    \   \rd k   \rd \vect{x} = : g(y).
  \end{align}
  This follows since applying the Casimir operator is done through right multiplication, thus, integrating with respect to $k$ reduces to the $K$-invariant case: i.e
  \begin{align*}
    \int_K \cC \wt{F}(ha_y k) \rd k = \Delta \int_K \wt{F}(ha_yk) \rd k.
  \end{align*}
  then we notice that we can use integration by parts to eliminate all the derivatives in $x_i$ coming from the Laplacian and since $\wt{F}$ is a compactly supported, smooth, function on $\Gamma \bk G / M$ it follows that all boundary terms are $0$. This leaves \eqref{diff eq}. 

  From here we apply a standard argument from the calculus of variations to conclude the following
  \begin{theorem}
    For any $s >0$ with $s \in (n/2,n)$, there exist constants $A, B$ such that
    \begin{align}
      f(y) = A y^{s} +  B y^{n-s} + y^{s} u(y) + y^{n-s} v(y)
    \end{align}
    where
    \begin{align*}
      &u(y) := (n-2s)^{-1}\int_{y}^{1/T+\vep/T^{n+1}} w^{n-s+2} g(w) \frac{\rd w}{w^{n+1}},
      \qquad \mbox{and}\\
      &v(y):= -(n-2s)^{-1}\int_{y}^{1/T+\vep/T^{n+1}} w^{s+2} g(y) \frac{\rd w}{w^{n+1}}
    \end{align*}
    Moreover, when $s = n/2+it$ then 
    \begin{align}
      f(y) = A y^{n/2} + By^{n/2}\log y + y^{n/2}u(y) + y^{n/2}v(y) \log y
    \end{align}
    where 
    \begin{align*}
      &u(y) := \int_{y}^{1/T+\vep/T^{n+1}} w^{n/2+2}\log(w) g(w) \frac{\rd w}{w^{n+1}} 
      \qquad \mbox{and}\\
      &v(y):= - \int_{y}^{1/T+\vep/T^{n+1}} w^{n/2+2} g(w) \frac{\rd w}{w^{n+1}}.
    \end{align*}
  \end{theorem}
  \begin{proof}
    The proof of this theorem is classical and the details can be found in \cite[Appendix B]{Kontorovich2009}. We omit the proof here.
  \end{proof}

  From here we can apply the integral in $y$ and use the same Cauchy-Schwarz argument, as used above to derive \eqref{M exp}, to write
  \begin{align}\label{M expand}
    \cM_{\vep}(T,F) = A \alpha(T) + B\beta(T) + O_{T,\vep,\Gamma}(\|(\cC-\lambda) \wt{F}\|_{L^2(\Gamma \bk G)} )
  \end{align}
  where, if $s \in (n/2,n)$
  \begin{align*}
    \alpha(T) &: =  \frac{1}{2\vep}\int_{1/T-\vep/T^{n+1}}^{1/T+\vep/T^{n+1}} y^{s-(n+1)} \rd y = T^{-s} + O(\vep T^{-s-1})\\
    \beta(T) &:=  \frac{1}{2\vep}\int_{1/T-\vep/T^{n+1}}^{1/T+\vep/T^{n+1}} y^{n-s-(n+1)} \rd y = T^{s-n} + O(\vep T^{s-n-1}),
  \end{align*}
  and if $s = n/2+it$
  \begin{align*}
    \alpha(T) &: = \frac{1}{2\vep} \int_{1/T-\vep/T^{n+1}}^{1/T+\vep/T^{n+1}} y^{n/2-(n+1)} \rd y = T^{-n/2} + O(\vep T^{-n/2-1})\\
    \beta(T) &:=  \frac{1}{2\vep}\int_{1/T-\vep/T^{n+1}}^{1/T+\vep/T^{n+1}} y^{n/2-s-(n+1)}\log(y) \rd y = T^{-n/2}\log(T) + O( \vep T^{-n/2-1}\log(T)).
  \end{align*}

  \subsection{Inserting the Laplacian}
  \label{ss:Inserting the Laplacian}

  Once again, fix a real number $b>0$, let
  \begin{align*}
    K_T(s) := \frac{\beta(b)\alpha(T) -\alpha(b)\beta(T)}{\alpha(1)\beta(b)-\alpha(b)\beta(1)}, \qquad
    L_T(s) := \frac{\alpha(1)\beta(T) -\beta(1)\alpha(T)}{\alpha(1)\beta(b)-\alpha(b)\beta(1)},
  \end{align*}
  and note that $K_T(s), L_T(s) \asymp T^{s-n}$ when $s\neq n/2$ and $K_T(s),L_T(s)\ll T^{-n/2}\log(T)$ when $s = n/2+it$. The following theorem (analogous to Theorem \ref{thm:main ident finite}) shows that we can 'grow' $\Psi_{\vep,1}$ to height $T$ using the Casimir operator

  \begin{theorem}
    For fixed $T\ge 1$ there exists a number $b>0$ and corresponding $K_T$ and $L_T$ such that
    \begin{align}\label{Psi ident}
      \Psi_{\vep,T} = K_T(\cC) \Psi_{\vep,1} + L_T(\cC) \Psi_{\vep,b}
    \end{align}
    almost everywhere. Moreover $K_T$ and $L_T$ satisfy the following bounds
    \begin{align}\label{KL bounds main}
      K_T(s), L_T(s) \ll \begin{cases}
        T^{s-n} & \mbox{ if } s \in (n/2,n)\\
        T^{-n/2}\log T & \mbox{ if } s = n/2+it.
       \end{cases}
    \end{align}
  \end{theorem}

  \begin{proof}
    First define the difference function
    \begin{align}\label{G def}
      G_T:=\Psi_{\vep,T} - K_T(\cC) \Psi_{\vep,1} + L_T(\cC) \Psi_{\vep,b}.
    \end{align}
    We can use exactly the argument in \cite[Proof of Proposition 3.5]{Kontorovich2009} to show that for any $\wt{F} \in L^2(\Gamma \bk G /M)$ we have that 
    \begin{align}\label{G bound}
      \<G_T,\wt{F}\> \ll_{\lambda,T} \|(\cC-\lambda) \wt{F} \|_{\Gamma}.
    \end{align}
    This follows by the construction of $K_T$ and $L_T$ and \eqref{M expand}, we omit the details. 
    
    Now, any function $G_T$ satisfying \eqref{G bound} for every $\lambda$ and any $\wt{F} \in L^2(\Gamma \bk G /M)$ must vanish almost everywhere by the same argument at \cite[Proof of Theorem 3.2]{Kontorovich2009}.
    
  \end{proof}

  \subsection{Proof of Theorem \ref{thm:main asymptotic}}

  Now, to prove Theorem \ref{thm:main asymptotic} we first exploit the abstract spectral theorem to pass to the spectrum. That is by the abstract Parseval's identity \eqref{API}
  \begin{align}
    \cM_{\vep}(T,F) &= \<\wt{F}, \Psi_{\vep,T}\>_{\Gamma} \notag\\
    &= \<\wh{\wt{F}}, \wh{\Psi_{\vep,T}}\>_{\Spec} \notag\\
    &=\wh{\wt{F}}(\lambda_0)\wh{\Psi_{\vep,T}}(\lambda_0)
    +\int_{\Spec \setminus\{\lambda_0 \}} \wh{\wt{F}}(\lambda) \wh{\Psi_{\vep,T}}(\lambda) \rd \wt{\mu}(\lambda).
  \end{align}

  First, let us analyze the term coming from the base eigenvalue. Using \eqref{proj} and our main identity \eqref{Psi ident} we have that
  \begin{align*} 
    \wh{\wt{F}}(\lambda_0)\wh{\Psi_{\vep,T}}(\lambda_0) &= K_T(\lambda_0)\< \Proj_{\mathscr{H}_0}(\wt{F}), \Proj_{\mathscr{H}_0}(\Psi_{\vep,1})\> + L_T(\lambda_0)\< \Proj_{\mathscr{H}_0}(\wt{F}), \Proj_{\mathscr{H}_0}(\Psi_{\vep,b})\>\\
    &=  T^{\delta-n}c \< \Proj_{\mathscr{H}_0}(\wt{F}), \Proj_{\mathscr{H}_0}(\Psi_{\vep,1}) + \Proj_{\mathscr{H}_0}(\Psi_{\vep,b}) \> + O(T^{-\delta}). 
  \end{align*}
  The projection $\Proj_{\mathscr{H}_0}(\Psi_{\vep,1})$ can be expressed using a Burger-Roblin type measure, as done in \cite[p. 861]{MohammadiOh2015}. From whence it follows that $\Proj_{\mathscr{H}_0}(\Psi_{\vep,1}) = C + O(\vep)$. Similarly, we can use our understanding of the base eigenspace to conclude that $\Proj_{\mathscr{H}_0}(\wt{F}) = C(1+O(\vep))$ for some $C$ depending only on $F$ (this comes from the smoothing in $\Kbar$ and another mean value theorem argument). From here, we conclude that
  \begin{align*} 
    \wh{\wt{F}}(\lambda_0)\wh{\Psi_{\vep,T}}(\lambda_0) 
    &=  cT^{\delta-n}(1+O(\vep)) + O(T^{-\delta}). 
  \end{align*}
  Note that we are unable to extract the $\vep$-dependence of the terms coming from larger eigenvalues. In thickened form (i.e $\vep$ fixed) we can apply the same process and extract lower order terms. However at present we have no way to extract the $\vep$-dependence of $\Proj_{\mathscr{H}_i}(\Psi_{\vep,1})$, since $\mathscr{H}_i$ has infinite multiplicity. Note that this is not the case if $F$ is right $K$-invariant, this is why, in that case, we can extract all lower order terms coming from the exceptional spectrum.

  As for the remainder of the spectrum we can apply Parseval and Cauchy-Schwarz in the same way as we did previously, arriving at
  \begin{align*}
    \int_{\Spec \setminus\{\lambda_0 \}} \wh{\wt{F}}(\lambda) \wh{\Psi_{\vep,T}}(\lambda) \ \rd \wt{\mu}(\lambda) \ll \max\{ T^{s_1-n}, T^{-n/2}\log T \} \|F\|_{\Gamma} \|\Psi_{\vep,1}\|_{\Gamma}
 \end{align*}
  Now we note that $\|\Psi_{\vep,1}\| = O(\vep^{-(\dim(\Kbar)+1)/2}) = O\left(\vep^{-\frac{n^2-3n+6}{4}}\right)$.

  From here deduce (let us assume for simplicity $s_1 < n/2$
  \begin{align*}
    \cM_{\vep}(T,F)= cT^{\delta-n}(1+O(\vep)) + O\left(\vep^{- \frac{n^2-3n+6}{4}}\|F\|_\Gamma T^{-n/2} \log T \right).
  \end{align*}
  By the same mean value principle argument as previously used
  \begin{align}
    \abs{\cM_{\vep}(T,F) - \cM_T(F)}\ll T^{\delta-n} \vep \|F\|_{1,\infty}.
  \end{align}
  Choosing $\vep = T^{-\frac{4(\delta-n/2) }{P(n)}} (\|F\|_\Gamma \|F\|_{1,\infty}^{-1} \log T)^{\frac{4}{P(n)}}$,
  \begin{align*}
    \cM_{\vep}(T,F)= cT^{\delta-n} + O\left((T^{-\frac{4(\delta-n/2)}{P(n)}} (\log T)^{\frac{4}{P(n)}}) T^{\delta-n} \|F\|_\Gamma^{\frac{4}{P(n)}} \|F\|_{1,\infty}^{\frac{P(n) -4}{P(n)}} \right).
  \end{align*}

\section{Preliminaries in higher rank}
\label{s:prelimRank}
Now let $G:= \SL_n(\R)$ for $n > 2$. 

\textbf{Decomposition of $L^2(\Gamma \bk G/K)$ into irreducibles and the spectral theorem:} Let $\Gamma < \SL_n(\Z)$ have finite co-volume, let $\mathscr{H}:= L^2(\Gamma \bk G/K)$ be the $L^2$ Hilbert space. The Casimir operators, $\Delta_1, \dots, \Delta_{n-1}$ are positive, self-adjoint operators acting on the $L^2$ Hilbert space. Thus the spectrum of each lies in $\R_{>0}$. We could apply the abstract spectral theorem to each individually, however, for our purposes this is not enough. We will need to consider the joint spectrum and spectral measure coming from the unitary dual representation. 

  The group $G$ acts by right regular representation on $\mathscr{H}$. The Hilbert space $\mathscr{H}$ decomposes into components as follows
  \begin{align*}
    \mathscr{H} = \mathscr{H}_0 \oplus \mathscr{H}_1  \oplus \dots  \oplus \mathscr{H}_k  \oplus \mathscr{H}^{tempered} 
  \end{align*}
  where $\mathscr{H}_i$ is a finite dimensional eigenspace with $\Delta_j$-eigenvalue $\lambda_j^{(i)}$ and $\mathscr{H}^{tempered}$ denotes the tempered spectrum.   

  Now there exists a measure $\wt{\mu}$ on the space of irreducible representations such that
  \begin{align*}
    f = \int \wh{f}(\vect{\lambda}) \rd \wt{\mu}(\vect{\lambda})
  \end{align*}
  (see \cite[p. 122]{CarterSegalMacDonald1995}), where $\wh{f}$ denotes the projection of $f$ onto the irreducible parameterized by $\vect{\lambda}$. Moreover, Harish-Chandra showed that this measure satisfies Plancherel's identity (see \cite[p. 122]{CarterSegalMacDonald1995})
  \begin{align}
    \int_{\Gamma \bk G} \abs{f}^2 \rd g = \int_{0}^\infty  | \wh{f} |^2 \rd \wt{\mu},
  \end{align}
  and by the polarization formula: $\<f,g\> = \frac{1}{2}(\|f\|^2 +\|g\|^2 -\|f-g\|^2)$, Plancherel's identity leads to Parseval's identity:
  \begin{align}\label{Parseval rank}
    \< f, g\>_{\Gamma} = \int_{0}^\infty  \wh{f}\ \wh{g}\ \rd \wt{\mu}.
  \end{align}
  Finally, by Schur's lemma \cite[p. 84]{CarterSegalMacDonald1995} the Casimir operators each act via multiplication on the irreducible representations: that is, given an irreducible representation we can associate a value $\vect{\lambda}$, then $\Delta_i$ acts via a scalar
  \begin{align} \label{diagonal}
    \wh{\Delta_i f}(\vect{\lambda}) = \lambda_i  \wh{f}(\vect{\lambda}).
  \end{align}
  Let $\cS$ denote the collection of scalars of $\Delta_1, \dots, \Delta_{n-1}$ acting on the different irreducible subspaces (without multiplicity). Thus
  \begin{align*}
    \cS = \{ \vect{\lambda}^{(0)},\vect{\lambda}^{(1)}, \dots, \vect{\lambda}^{(k)}\} \cup \cS^{cont},
  \end{align*}
  where $\cS^{cont}$ denotes the tempered spectrum.

  \section{Expanding Horospheres in $\SL_3(\R)$}
  \label{s:SL3R}

  In higher rank, the proof follows the same approach with some modification; for simplicity we start with the case $n=3$. Once again we need to smooth in the non-horospherical directions, in this case $y_1$ and $y_2$. To that end, fix a value $\vep$ (one could choose to have two smoothing parameters, but when optimizing it turns out that the optimal choice is to set them equal to each other). Let $\psi_{T,\vep}$ be the indicator of the region $[1/T-\vep/T^3,1/T+\vep/T^3]$ multiplied by $\vep^{-1}$. For $\cH \ni z = h(\vect{x})a(\vect{y})$ we write 
  \begin{align*}
    \psi_{\vect{T},\vep}(z) = \psi_{T_1,\vep}(y_1) \psi_{T_2,\vep}(y_2), 
  \end{align*}
  now automorphize $\psi_{\vect{T},\vep}$ as follows
  \begin{align*}
    \Psi_{\vect{T},\vep}(z) = \sum_{\Gamma_H \bk \Gamma}\psi_{\vect{T},\vep}(\gamma z).
  \end{align*}
  Let $\cF$ denote a fundamental domain for the group $\Gamma$. Now we define the $\vep$-thickened horospherical average as
  \begin{align*}
    \cM_{\vep}(\vect{T},F): = \int_{\cF}F(z) \Psi_{\vect{T},\vep}(z) \rd z.
  \end{align*}

  \subsection{Differential Equation}

  Once again, our first step is to unfold $\cM_{\vep}$:
  \begin{align*}
    \cM_{\vep}(\vect{T},F) &= \int_{\Gamma_H\bk \cH} F(z) \psi_{\vep,\vect{T}}(z) \rd z\\
    &= \int_{\R^2_{>0}}\psi_{\vep,T_1}(y_1)\psi_{\vep,T_2}(y_2) \int_{\Gamma_H\bk H} F(h a(\vect{y})) \rd h \frac{\rd y_1\rd y_2}{y_1^3y_2^3}.
  \end{align*}
  Let $f(\vect{y}):= \int_{\Gamma_H\bk H} F(h a(\vect{y})) \rd h$ denote the inner integral.

  Fix any vector $\vect{\lambda}=(\lambda_1, \lambda_2)$.  Then by integration by parts,  the function $f$ satisfies the equation
  \begin{align}\label{diffeq 1}
    (y_1^2\partial_{y_1y_1} +y_2^2\partial_{y_2y_2}-y_1y_2\partial_{y_1y_2}-\lambda_1) f(\vect{y}) = \int_{\Gamma_H\bk H} (\Delta_1-\lambda_1)F(h a(\vect{y})) \rd h = : g_1(y_1,y_2).
  \end{align}
  Likewise, $f$ satisfies
  \begin{align}\label{diffeq 2}
    (-y_1^2y_2\partial_{y_1y_1y_2} + y_1 y_2^2\partial_{y_1y_2y_2} + y_1^2\partial_{y_1y_1}-y_2^2\partial_{y_2y_2}  -\lambda_2) f(\vect{y}) = \int_{\Gamma_H\bk H} (\Delta_2-\lambda_2)F(h a(\vect{y})) \rd h.
  \end{align}
  In both cases $y_1^{s}y_2^{r}$ is a solution to the homogeneous case, but the first equation requires $sr-s(s-1)-r(r-1)= \lambda_1$ and the second requires $sr(s-r)  +s(s-1) +r(r-1)  = \lambda_2$. For a given pair $\lambda_1,\lambda_2$ there are at most six pairs $(s, r)$ satisfying both equations. Let $\phi_{homo}(\vect{y}):=  \sum_{i=1}^6 A_i y_1^{s_i}y_2^{r_i}$ denote the homogeneous solution to \emph{both} equations.

  Let $\phi_{homo,1}$ denote a general solution to the homogeneous \eqref{diffeq 1}. Now we can use the method of Green's functions to lift this homogeneous solution to a solution of the inhomogeneous equation, that is our particular solution will take the form
  \begin{align*}
    u_p(y_1, y_2) : = \int_{\xi_1,\xi_2} \cG(y_1,y_2,\xi_1,\xi_2) g(\xi_1,\xi_2) \rd \xi_1\rd \xi_2.
  \end{align*}
  Then $\cG$ satisfies the equation
  \begin{align*}
    (y_1y_2\partial_{y_1y_2} - y_1^2\partial_{y_1y_1} -y_2^2\partial_{y_2y_2}-\lambda)\cG(y_1,y_2,\xi_1,\xi_2) = \delta(y_1-\xi_1) \delta(y_2-\xi_2).
  \end{align*}
  Now assume $\cG$ is point pair invariant, i.e it depends only on the scalar $r=\abs{\vect{\xi}-\vect{y}}$. In this case, after integrating out the angular direction the above equation becomes for $r>0$
  \begin{align*}
    \frac{\pi}{2}r^2 \partial_{rr} h(r) = 2\pi \lambda_1 h(r).
  \end{align*}
  Hence, writing $\lambda_1= \kappa (\kappa-1)$ yields   
  \begin{align*}
    \cG(\vect{y},\vect{\xi}) = \abs{\vect{\xi}-\vect{y}}^{\kappa}.
  \end{align*}
  Thus our full solution is
  \begin{align}
    f(\vect{y}) = \phi_{homo, 1}(\vect{y}) + u_{p,1}(y_1,y_2).
  \end{align}
  Similarly, if we apply the same argument with $\Delta_2$ in place of $\Delta_1$ we arrive at the equation
  \begin{align}
    f(\vect{y}) = \phi_{homo, 2}(\vect{y}) + u_{p,2}(y_1,y_2),
  \end{align}
  where $\phi_{homo,2}$ satisfies $\Delta_2 \phi_{homo, 2} =0$ and $u_{p,2}(y_1,y_2)$ is the particular solution associated to $\lambda_2$.

  Thus, we can write
  \begin{align*}
    f(\vect{y}) &= \phi_{homo}(\vect{y}) + \wt{\phi}_{homo, 1}(\vect{y}) + u_{p,1}(y_1,y_2)\\
    &= \phi_{homo}(\vect{y}) + \wt{\phi}_{homo, 2}(\vect{y}) + u_{p,2}(y_1,y_2),
  \end{align*}
  where $\wt{\phi}_{homo, i}(\vect{y}) = \phi_{homo, i}(\vect{y}) - \phi_{homo}(\vect{y})$, for $i=1,2$.
  Now, to integrate with respect to $\vect{y}$, let
  \begin{align*}
    \alpha(\vect{T}) : =\int_0^\infty \int_0^\infty \psi_{\vep,\vect{T}}(\vect{y})\phi_{homo}(\vect{y}) y_1^{-3}y_2^{-3} \rd y_1\rd y_2.
  \end{align*}
  Then we have that, by the usual Cauchy-Schwartz argument and the definitions of $\phi_{homo},$ $\phi_{homo,1},$ and $\phi_{homo,2}$, for every pair $\lambda_1$ and $\lambda_2$ in $\R^2$ we have
  \begin{align}\label{M Delta bound}
    \cM_{\vep}(\vect{T},F) = \alpha(\vect{T}) + O(\|(\Delta_1 - \lambda_1)F\|) + O(\|(\Delta_2 - \lambda_2)F\|).
  \end{align}

  Further, for $i=1,\dots,6$, let 
  \begin{align}
    \alpha_i(\vect{T}) &: =\int_0^\infty \int_0^\infty \psi_{\vep,\vect{T}}(\vect{y}) y_1^{s_i}y_2^{r_i} y_1^{-3}y_2^{-3} \rd y_1\rd y_2 \notag \\
    &= T_1^{-s_i}T_2^{-r_i} + O(\vep T_1^{-s_i-1}T_2^{-r_i} + \vep T_2^{-r_i-1}T_1^{-s_i}) .\label{alpha estimate}
  \end{align}
  Then
  \begin{align}
    \cM_{\vep}(\vect{T},F) = \sum_{i=1}^6A_i\alpha_i(\vect{T}) + O(\|(\Delta_1 - \lambda_2)F\|) + O(\|(\Delta_2 - \lambda_2)F\|).
  \end{align}
  Now we fix $6$ times $\vect{b}_i\in \R^2_{>1}$ and write the matrix equation
  \begin{align*}
      (\cM_{\vep}(\vect{b}_1,F),
      \cM_{\vep}(\vect{b}_2,F), \dots
      \cM_{\vep}(\vect{b}_6,F))^T
    =
    \vect{M} \vect{A}.
  \end{align*}
  where $(\vect{M})_{ij} = \alpha_i(\vect{b}_j)$ and $\vect{A} = (A_1,A_2, \dots, A_6)^T$. By an appropriate choice of $\vect{b}_i$ we can ensure that $M$ is invertible. From here we can solve for $A_i$ and write
  \begin{align}\label{M ident pre}
    \cM_{\vep}(\vect{T},F) = \sum_{i=1}^6 K_{i,\vect{T}}(\vect{\lambda}) \cM_{\vep}(\vect{b}_i, F)  + O(\|(\Delta_1 - \lambda_1)F\|) + O(\|(\Delta_2 - \lambda_2)F\|)
  \end{align}
  for some $K_{i,\vect{T}}(\vect{\lambda})$ which can be made explicit. Now we are ready to prove following main identity which underpins the spectral estimates that are crucial in our later analysis. It states that in \eqref{M ident pre} we can replace the functions of $\vect{\lambda}$ with the same functions of $\vect{\Delta} = (\Delta_1, \Delta_2)$ (defined via power series) without affecting the error term. Then since the modified \eqref{M ident pre} holds for all points $\vect{\lambda}$ in the spectrum, we can in fact, show that the error vanishes.

    \begin{theorem}[Main Identity]
    There exist constants $\vect{b}_i$  such that for $\vect{T}$ large enough
    \begin{align}\label{main ident SL3R}
      \Psi_{\vep,\vect{T}} = \sum_{i=1}^6 K_{i,\vect{T}}(\vect{\Delta})\Psi_{\vep,\vect{b}_i},
    \end{align}
    almost everywhere. Moreover $K_{i,\vect{T}}$ all satisfy,
    \begin{align}\label{K bounds}
      K_{i,\vect{T}}(\vect{\lambda}) = \begin{cases}
         T_1^{-s_i}T_2^{-r_i} + O(\vep T_1^{-s_i-1}T_2^{-r_i}+\vep T_1^{-s_i}T_2^{-r_i-1})
        & \text{ if } s_i,r_i  < 1  \\
        T_1^{-1}T_2^{-1}\log(T_1)\log(T_2) & \text{ if } s_i,r_i \in \{ 1+it\}^2.
        \end{cases}
    \end{align}
  \end{theorem}

  \begin{proof}

    First, note that for any $i$ we can apply the spectral transform and the mean value theorem to
    \begin{align*}
      \<K_{i,\vect{T}}(\vect{\Delta}) \Psi_{\vep,\vect{b}}, F \> - K_{i,\vect{T}}(\vect{\lambda})\< \Psi_{\vep,\vect{b}},F\> &= \int_{\cS} (K_{i,\vect{T}}(\vect{\lambda}^\prime) - K_{i,\vect{T}}(\vect{\lambda})) \wh{\Psi_{\vep,\vect{b}}}(\vect{\lambda}^\prime) \wh{F}(\vect{\lambda}^\prime) \ \rd \wt{\mu}(\vect{\lambda}')\\
      &\ll \int_{\cS} \abs{\vect{\lambda}^\prime - \vect{\lambda}} \wh{\Psi_{\vep,\vect{b}}}(\vect{\lambda}^\prime) \wh{F}(\vect{\lambda}^\prime)  \rd \wt{\mu}(\vect{\lambda}')\\
      &\ll \int_{\cS} (\abs{\lambda_1^\prime - \lambda_1}+\abs{\lambda_2^\prime - \lambda_2}) \wh{\Psi_{\vep,\vect{b}}}(\vect{\lambda}^\prime) \wh{F}(\vect{\lambda}^\prime)  \rd \wt{\mu}(\vect{\lambda}')  \\
      &\ll \|(\lambda_1 - \Delta_1) F\|\| \Psi_{\vep,\vect{b}}\| + \|(\lambda_2 - \Delta_2) F\|\| \Psi_{\vep,\vect{b}}\|.
    \end{align*}
    From here we can use the definition of $K_{i,\vect{T}}$ to conclude (see \cite[Proof of Lemma 3.5]{Kontorovich2009} for details)
    \begin{align*}
      \cM_{\vep}(\vect{T})  =  \sum_{i=1} K_{i,\vect{T}}(\vect{\Delta}) \cM_{\vep}(\vect{b}_i,F) + O(\|(\Delta_1 - \lambda_1)F\|) + O(\|(\Delta_2 - \lambda_2)F\|).
    \end{align*}
    Define for $z \in \cH$
    \begin{align*}
      G_{\vep,\vect{T}}(z) : = 
      \Psi_{\vep,\vect{T}}(z)  -  \sum_{i=1} K_{i,\vect{T}}(\vect{\Delta}) \Psi_{\vep,\vect{b}_i}(z).
    \end{align*}
    Then \eqref{main ident SL3R} will follow if we can show $G_{\vep,\vect{T}}=0$ almost everywhere.

    Now for any function $F$ and any point in the spectrum $\vect{\lambda}$ we have
    \begin{align*}
      \<G_{\vep,\vect{T}}, F \> =O(\|(\Delta_1 - \lambda_1)F\|) + O(\|(\Delta_2 - \lambda_2)F\|).
    \end{align*}
    Fix a $\sigma>0$ and a point $\vect{\lambda}$  and define $F$ according to its spectrum:
    \begin{align*}
      \wh{F}(\vect{\lambda}^\prime) :=
      \begin{cases}
        \wh{G_{\vep,\vect{T}}}(\vect{\lambda}^\prime) & \mbox{ if } \vect{\lambda}^\prime \in B_\sigma(\vect{\lambda}) \\
        0 & \mbox{ otherwise,} 
      \end{cases}
    \end{align*}
    where $\wh{G_{\vep,\vect{T}}}$ denotes the spectral transform of $G_{\vep,\vect{T}}$, and $B_\sigma(\vect{\lambda})$ denotes a ball of radius $\sigma$ around $\vect{\lambda}$. Now apply Parseval's identity \eqref{Parseval rank}
    \begin{align*}
      \< G_{\vep,\vect{T}},F\>& = \<\wh{G_{\vep,\vect{T}}},\wh{F}\>\\
             &= \int_{B_{\sigma}(\vect{\lambda})}\abs{\wh{G_{\vep,\vect{T}}}(\vect{\lambda}^\prime)}^2 \rd \wt{\mu}(\vect{\lambda}^\prime).
    \end{align*}
   Using \eqref{diagonal} yields
    \begin{align*}
      \|(\Delta_1-\lambda_1)F\| &\ll \left(\int_{0}^\infty\int_0^\infty \abs{(\lambda_1^\prime - \lambda_1)\wh{F}(\vect{\lambda})}^2 \rd \wt{\mu}(\vect{\lambda})\right)^{1/2}\\
        &\ll \sigma \left(\int_{B_\sigma(\vect{\lambda})} \abs{\wh{G_{\vep,\vect{T}}}(\vect{\lambda}^\prime)}^2 \rd \wt{\mu}(\vect{\lambda}^\prime)\right)^{1/2},
    \end{align*}
    and similarly
    \begin{align*}
      \|(\Delta_2-\lambda_2)F\|
      &\ll \sigma \left(\int_{B_\sigma(\vect{\lambda})} \abs{\wh{G_{\vep,\vect{T}}}(\vect{\lambda}^\prime)}^2 \rd \wt{\mu}(\vect{\lambda}^\prime)\right)^{1/2}.
    \end{align*}
    Thus, for any $\sigma$ we have
    \begin{align*}
      \int_{B_{\sigma}(\vect{\lambda})}\abs{\wh{G_{\vep,\vect{T}}}(\vect{\lambda}^\prime)}^2 \rd \wt{\mu}(\vect{\lambda}^\prime) \ll  \sigma \left(\int_{B_\sigma(\vect{\lambda})} \abs{\wh{G_{\vep,\vect{T}}}(\vect{\lambda}^\prime)}^2 \rd \wt{\mu}(\vect{\lambda}^\prime) \right)^{1/2}.
    \end{align*}
    If the left hand side vanishes we are done. If not, then
    \begin{align}
      \int_{B_\sigma(\vect{\lambda})} \abs{\wh{G_{\vep,\vect{T}}}(\vect{\lambda}^\prime )}^2 \rd \wt{\mu}(\vect{\lambda}^\prime) \ll \sigma^2
    \end{align}
    for any $\vect{\lambda}$ and $\sigma$. Now  let
    \begin{align*}
    f(\vect{\lambda}): = \int_{\lambda_1^\prime < \lambda_1}\int_{\abs{\lambda_2^\prime - \lambda_2}<\sigma} \abs{\wh{G_{\vep,\vect{T}}}(\vect{\lambda}^\prime )}^2 \rd \wt{\mu}(\vect{\lambda}^\prime)
    \end{align*}
    then we can take a derivative
    \begin{align*}
     \frac{\rd f}{\rd \lambda_1}(\lambda_1,\lambda_2) = \frac{f(\lambda_1+\sigma,\lambda_2)-f(\lambda_1-\sigma,\lambda_2)}{2\sigma} = \lim_{\sigma \to 0} \frac{1}{2\sigma} \int_{X} \abs{\wh{G}_{\vep,\vect{T}}(\lambda_1^\prime)}^2 \rd \wt{\mu}(\lambda_1^\prime) = 0.
    \end{align*}
    Thus $f^\prime(\vect{\lambda}) =0$ for all $\vect{\lambda}$, and since $f(0)=0$ we conclude that $f=0$ for all values of $\vect{\lambda}$.

    Now let
    \begin{align*}
      \wt{f} (\vect{\lambda}) = \int_{\lambda_1^\prime < \lambda_1}\int_{\lambda_2^\prime <  \lambda_2} \abs{\wh{G_{\vep,\vect{T}}}(\vect{\lambda}^\prime )}^2 \rd \wt{\mu}(\vect{\lambda}^\prime)
    \end{align*}
    then it is easy to show that $ \frac{\rd}{\rd \lambda_2}\wt{f}(\vect{\lambda}) =0$ for all $\vect{\lambda}$ and $\wt{f}(0)=0$. Similarly we can deduce that $ \frac{\rd}{\rd \lambda_1}\wt{f}(\vect{\lambda}) =0$ for all $\vect{\lambda}$. Thus $\wt{f}(\vect{\lambda})=0$ for all $\vect{\lambda}$. From which it follows that $\wh{G}_{\vep,\vect{T}}$ is $0$ almost everywhere.

  \end{proof}

  \subsection{Proof of Theorem \ref{thm:SL3R}}

  Now, with the main identity at hand, we can proceed with the proof of Theorem \ref{thm:SL3R}.  By Parseval's identity \eqref{Parseval rank}  
  \begin{align*}
   \cM_{\vep}(\vect{T},F) &= \< \Psi_{\vep,\vect{T}}, F\>_\Gamma\\
   &= \< \wh{\Psi_{\vep,\vect{T}}}, \wh{F} \>_{\Spec(\Gamma)}\\
   &= \wh{F}(\vect{\lambda}^{(0)}) \wh{\Psi_{\vep,\vect{T}}}(\vect{\lambda}^{(0)}) + \wh{F}(\vect{\lambda}^{(1)}) \wh{\Psi_{\vep,\vect{T}}}(\vect{\lambda}^{(1)}) + \dots + \wh{F}(\vect{\lambda}^{(k)}) \wh{\Psi_{\vep,\vect{T}}}(\vect{\lambda}^{(k)}) \\
   &\phantom{++++++++++++++}+ \int_{\cS^{cont}} \wh{F}(\vect{\lambda}) \wh{\Psi_{\vep,\vect{T}}}(\vect{\lambda}) \rd\wt{\mu}(\vect{\lambda}).
  \end{align*}
  Since $F$ has compact support and the $i^{th}$ eigenspace has finite dimension, we know that the projection onto the $i^{th}$ eigenspace,
  \begin{align*}
    \wh{F}(\vect{\lambda}^{(i)}) = \sum_{k=1}^\kappa\< F, \phi_{i,k}\>
  \end{align*}
  is a finite constant, where $\kappa$ is the multiplicity of the $i^{th}$ eigenspace. Furthermore, using our bounds for $K_{i,\vect{T}}$, \eqref{K bounds} we have
 \begin{align*}
   \wh{\Psi_{\vep,\vect{T}}}(\vect{\lambda}^{(i)}) = \sum_{k=1}^\kappa\left(\sum_{j=1}^6 c_j  T_1^{-s_j^{(i)}}T_2^{-r_j^{(i)}} \<\Psi_{\vep,\vect{1}},\phi_{i,k}\>\right) 
   (1+O(\vep T_1^{-1} + \vep T_2^{-1})). 
 \end{align*}
 Moreover, by the same mean value argument we have that $\<\Psi_{\vep,\vect{1}},\phi_{i,k}\> =  C+O(\vep)$. From whence it follows that
 \begin{align*}
   \wh{F}(\vect{\lambda}^{(i)}) \wh{\Psi_{\vep,T}}(\vect{\lambda}^{(i)}) = C_{\Gamma,F} \left(\sum_{j=1}^6 c_j  T_1^{-s_j^{(i)}}T_2^{-r_j^{(i)}}\right)(1+O(\vep)).
 \end{align*}
 
 Turning to the error term, we apply the fact that the Casimir operators are diagonal on the dual space
 \begin{align*}
   &\int_{\cS^{cont}} \wh{F}(\vect{\lambda}) \wh{\Psi_{\vep,\vect{T}}}(\vect{\lambda}) \rd\wt{\mu}(\vect{\lambda})\\
   &\phantom{+++++}
   =
   \int_{\cS^{cont}} \wh{F}(\vect{\lambda}) \left(\sum_{i=1}^6 \wh{K_{i,\vect{T}}(\vect{\Delta})\Psi_{\vep,\vect{b}_i}}(\vect{\lambda})\right) \rd\wt{\mu}(\vect{\lambda})\\
   &\phantom{+++++}
   =
   \int_{\cS^{cont}} \wh{F}(\vect{\lambda}) \left(\sum_{i=1}^6 K_{i,\vect{T}}(\vect{\lambda})\wh{\Psi_{\vep,\vect{b}_i}}(\vect{\lambda})\right) \rd\wt{\mu}(\vect{\lambda})\\
   &\phantom{+++++}
   \ll
   T_1^{- 1}T_2^{-1} \log T_1 \log T_2\int_{\cS^{cont}} \wh{F}(\vect{\lambda}) \left(\sum_{i=1}^6\wh{\Psi_{\vep,\vect{b}_i}}(\vect{\lambda})\right) \rd\wt{\mu}(\vect{\lambda}).
  \end{align*}
  To conclude we apply the abstract Parseval's identity and Cauchy-Schwarz giving
  \begin{align*}
    \int_{\cS^{cont}} \wh{F}(\vect{\lambda}) \wh{\Psi_{\vep,\vect{T}}}(\vect{\lambda}) \rd\wt{\mu}(\vect{\lambda})
   &\ll T_1^{- 1}T_2^{-1} \log T_1\log T_2 \|F\|_{\Gamma} \| \Psi_{\vep,\vect{1}}\|_\Gamma.
  \end{align*}
  Since $F$ is assumed to be in $L^2(\Gamma \bk \cH)$ its norm is bounded. As for the second term we have
  \begin{align*}
    \| \Psi_{\vep,\vect{1}}\|_\Gamma \ll \frac{1}{\vep}.
  \end{align*}
  Leading to the following 'thickened' version of the equidistribution result 
  \begin{align}\label{rank thick infinite}
    \begin{aligned}
      &\cM_{\vep}(\vect{T},F) = c_0m_{\vect{T}}(\vect{\lambda}^{(0)}) + c_1 m_{\vect{T}}(\vect{\lambda}^{(1)})+\dots + c_k m_{\vect{T}}(\vect{\lambda}^{(k)})\\
      &\phantom{+++++++++++++++++++}+ O(\vep^{-1}T_1^{-1}T_2^{-1} \log T_1 \log T_2  \|F\|_{\Gamma}    ),
    \end{aligned}
  \end{align}
  where, given a $\vect{\lambda}$ in the spectrum $m_{\vect{T}}(\lambda_1):= \sum_{i=1}^6 T_1^{s_i}T_2^{r_i}$. Note that the $c_i$ depend on the group, $\vep$ and on $F$. Moreover for $i =0, \dots, k$ we can write $c_i = C_i+O(\vep)$ where $C_i$ are independent of $\vep$.

  \subsection{Proof of Theorem \ref{thm:SL3R}}
  \label{ss:Proof of SL3R}

  Finally, we need to estimate the difference $\abs{\cM_{\vep}(\vect{T},F) - \cM(\vect{T},F)}$. Now we apply the mean value theorem to conclude
  \begin{align*}
    \abs{\cM_{\vep}(\vect{T},F) - \cM(\vect{T},F)} \ll \vep \|F\|_{1,\infty}.
  \end{align*}
  Hence we choose $\vep$ to maximize the error terms, and note that $m_{\vect{T}}(\vect{\lambda}^{(0)}) \asymp 1$, leading to
  \begin{align*}
    \vep^{-1}T_1^{-1}T_2^{-1} \log T_1 \log T_2 \|F\|_{\Gamma} = \vep\|F\|_{1,\infty}.
  \end{align*}
  Solving for $\vep$ then yields
  \begin{align*}
    \vep =  T_1^{-1/2}T_2^{-1/2} (\log T_1 \log T_2)^{1/2} \|F\|_{\Gamma}^{1/2}\|F\|_{1,\infty}^{-1/2}.
  \end{align*}

  \section{Expanding Horospheres $\SL_n(\R)$}
  \label{s:SLnR}

  Finally, we extend our proof to all $n\ge 2$. For for general $n$ the proof is almost identical to the proof for $\SL_3(\R)$; the only essential input is a particular structure of the Casimir operators, $\Delta_1, \dots, \Delta_{n-1}$ summarized in the following theorem:
  
  \begin{theorem}[Structure Theorem of Casimir Operators]
    For $\SL_n(\R)$, when acting on a right $K$-invariant function, $f\in L^2(\Gamma \bk \cH)$,  the Laplace-Beltrami operator satisfies
    \begin{align}\label{LB form}
      \int_{\Gamma \bk H} \Delta_1 f( h a_{\vect{y}}) \rd h
      = \left(\sum_{i=1}^{n-1} y_i^2 \partial_{y_1y_1} - \sum_{i=1}^{n-2} y_i y_{i+1} \partial_{y_i y_{i+1}} \right) \int_{\Gamma \bk H} f( h a_{\vect{y}}) \rd h.
    \end{align}
    Furthermore, for each Casimir operator, $\Delta_i$ there exists a differential operator $D_{i}$ in the $\vect{y}$ variables, such that
    \begin{align} \label{Cas form}
      \int_{\Gamma \bk H} \Delta_i f( h a_{\vect{y}}) \rd h
      = D_i \int_{\Gamma \bk H} f( h a_{\vect{y}}) \rd h.
    \end{align}
    
  \end{theorem}

  \begin{proof}

    We begin with an example for $n= 4$. In that case, if we integrate out the $\vect{x}$ coordinates, then using Goldfeld's {\tt gln} Mathematica package \cite{Goldfeld2006} we have that the Laplace-Beltrami operator, after integrating out the $\vect{x}$-directions, in $n=4$ is given by
  \begin{align*}
    \Delta_1 f =\left( 
    y_1^2\partial_{y_1y_1}
    +y_2^2\partial_{y_2y_2}    
    +y_3^2\partial_{y_3y_3}
    -y_2 y_1\partial_{y_1y_2}
    -y_2 y_3\partial_{y_2y_3}
    \right)f.
  \end{align*}
  The degree $3$ operator is:
  \begin{align*}
    \Delta_2 f = (
    &
    y_1^2 \partial_{y_1y_1}
    -3 y_2 y_1^2 \partial_{y_1y_1y_2}
    -7 y_2 y_1 \partial_{y_1y_2}
    +3 y_2^2 y_1 \partial_{y_1y_2y_2}
    -2 y_3^2\partial_{y_3y_3}
    \\
    &
    \phantom{+++++++}
   -y_2 y_3 \partial_{y_2y_3}
   +3 y_2 y_3^2 \partial_{y_2y_3y_3}
   +4 y_2^2\partial_{y_2y_2}
   -3 y_2^2y_3 \partial_{y_2y_2y_3}
   ) f
  \end{align*}
  and the degree $4$ Casimir is then
  \begin{align*}
    \Delta_3 & f = (y_1^4\partial_{y_1}^4 +4
    y_1^3\partial_{y_1}^3   -2 y_1^3 y_2 \partial_{y_1}^3\partial_{y_2}
    +21 y_1^2 \partial_{y_1}^2
    -12 y_2 y_1^2 \partial_{y_1}^2\partial_{y_2}
    +3 y_2^2 y_1^2 \partial_{y_2}^2\partial_{y_1}^2\\
    &
    -9 y_2 y_1 \partial_{y_2}\partial_{y_1}
    +6 y_2^2 y_1\partial_{y_1}\partial_{y_2}^2
    -2 y_2^3 y_1 \partial_{y_2}^3 \partial_{y_1}
    -3 y_3^2\partial_{y_3}^2
    +4 y_3^3 \partial_{y_3}^3
    +y_3^4 \partial_{y_3}^4
    +3 y_2 y_3 \partial_{y_2y_3}\\
    &
    -2 y_2 y_3^3 \partial_{y_2}\partial_{y_3}^3
    +3 y_2^2 \partial_{y_2}^2
    -6 y_2^2 y_3 \partial_{y_2}^2 \partial_{y_3}
    +3 y_2^2 y_3^2 \partial_{y_2}^2 \partial_{y_3}^2
    +4 y_2^3 \partial_{y_2}^3
   -2 y_2^3 y_3 \partial_{y_2}^3\partial_{y_3}
   +y_2^4 \partial_{y_2}^4
   )f.
   \end{align*}
   
  In general, we can calculate the Casimir operators as in \cite[Proposition 2.3.3]{Goldfeld2006}. Then analyzing the differential operators from \cite[Definition 2.2.1]{Goldfeld2006} and applying an inductive argument on their products is enough to prove the \eqref{LB form} and \eqref{Cas form}. That is, any element of the lie algebra, $X$ we can associate a first order differential operator coming from
  \begin{align*}
    T_X f = \frac{\rd}{\rd t} f(g e^{tX}) \bigg|_{t=0}.
  \end{align*}
  Thus $T_X$ has the form
  \begin{align*}
    T_X = \eta_1 \partial_{x_1} + \dots + \eta_{n(n-1)/2} \partial_{x_{n(n-1)/2}} + \mu_1 \partial_{y_1} + \dots + \mu_{n-1}\partial_{y_{n-1}}.
  \end{align*}
  Since we are multiplying on the right, it is evident that for the differential operators associated to $X_{y_k}$ the coefficients $\eta_i=0$ for $i = 1, \dots n(n-1)/2$ and the $\mu_j$ do not depend on $\vect{x}$.

  Moreover for the differential operators associated to $X_{x_i}$, it is not hard to see (by matrix multiplication rules) that $\eta_i$ does not depend on $x_i$ and $\mu_i=0$ for $i = 1, \dots, n-1$. From here \eqref{Cas form} follows via integration by parts. The same holds for the dual elements $\overline{X_{x_i}}$. The Laplace-Beltrami operator can be found in numerous places, and \eqref{Cas form} follows from a similar matrix multiplication argument. 

  \end{proof}

  Once again, fix an $\vep>0$, the proof begins by thickening the $\vect{y}$ directions. Recall that the volume measure is
  \begin{align*}
    \rd \vect{x} \prod_{k=1}^{n-1} y_k^{-k(n-k)-1} \rd y_k
  \end{align*}
  Therefore let
  \begin{align*}
    \psi^{(i)}_{T_i,\vep} = \vep^{-1} \one ([1/T- \vep/T^{k(n-k)+1}, 1/T+ \vep/T^{k(n-k)+1}]).
  \end{align*}
  For $z\in \cH$ let
  \begin{align*}
    \psi_{\vect{T},\vep}(z) : = \prod_{i=1}^{n-1} \psi^{(i)}_{T_i,\vep}(y_i)
  \end{align*}
  and let $\Psi_{\vect{T},\vep}:= \sum_{\Gamma_H \bk \Gamma} \psi_{\vect{T},\vep}(\gamma z)$. Then our thickened average is
  \begin{align*}
    \cM_{\vep} (\vect{T},F):=\int_{\cF} F(z) \Psi_{\vect{T}, \vep}(z) \rd z,
  \end{align*}
  where $\cF$ is a fundamental domain for $\Gamma$.

  With that, we apply the same unfolding steps and set
  \begin{align*}
    f(\vect{y}) : = \int_{\Gamma_H \bk H} F(ha(\vect{y})) \rd h
  \end{align*}
  then the  analogue of \eqref{diffeq 1} and \eqref{diffeq 2} is: for any $\vect{\lambda}$ 
  \begin{align}
    D_i f(\vect{y}) = \int_{\Gamma_H \bk H} (\Delta_i - \lambda_i) F(h a(\vect{y})) \rd h,
  \end{align}
  for $i =1 , \dots, n-1$.
  Which, by moving to the unitary dual and using the fact that the Casimir elements are diagonal, yields the analogue of \eqref{M ident pre}, that there exist functions $K_{i,\vect{T}}(\vect{\lambda})$ and times $\vect{b}_i$ such that we can express
  \begin{align*}
    \cM_{\vep}(\vect{T})  = \sum_{i=1}^{L}K_{i,\vect{T}}(\vect{\lambda}) \cM_{\vep}(\vect{b}_i,F) + \sum_{i=1}^{n-1} O(\| (\Delta_i - \lambda_i)F\|)
  \end{align*}
  recall that $L=n!$ denotes the number of solutions $\vect{\nu}$ associated to the point $\vect{\lambda}$. From whence we can derive the main identity
  \begin{align}\label{main ident SLnR}
    \Psi_{\vep,\vect{T}} = \sum_{i=1}^{L} K_{i, \vect{T}}(\vect{\Delta}) \Psi_{\vep, \vect{b}_i}.
  \end{align}
  Moreover we have that 
  \begin{align*}
    K_{i, \vect{T}} (\vect{\lambda} ) \ll \begin{cases}
      I_{\vect{\nu}_i}(\vect{T})  & \mbox{ if } \vect{\nu}_i \in [0,1/n)^{n-1}\\
      I_{cont}(\vect{T}) \log(T_1) \cdots \log T_{n-1} & \mbox{ if } \nu_{i,j} = 1/n +i t_j \mbox{ for all } j.
    \end{cases}
  \end{align*}
  From here, the remainder is a somewhat trivial generalization of Section \ref{s:SL3R}. Namely, after applying the spectral decomposition and \eqref{main ident SLnR} we arrive at the thickened equidistribution result
  \begin{align}
    \cM_{\vep}(T,F) = \sum_{i=0}^k c_i m_{\vect{T}}(\vect{\lambda}^{(i)}) + O(\vep^{-(n-1)/2} (I_{cont}(\vect{T}))^{-1} \log T_1 \cdots \log T_{n-1} \|F\|_\Gamma ).
  \end{align}
  The mean value theorem argument yields an error of size
  \begin{align*}
    \vep m_{\vect{T}}(\vect{\lambda}^{(0)})\|F\|_{1,\infty}.
  \end{align*}
  Setting these errors equal to each other, and using that $m_{\vect{T}}(\vect{\lambda}_0) \asymp 1$ we arrive at 
  \begin{align*}
    \vep  = \left(I_{cont}(\vect{T})\|F\|_{\Gamma}\|F\|_{1,\infty}^{-1}  \log T_1 \cdots \log T_{n-1}\right)^{2/(n+1)}
  \end{align*}
  Thus our error becomes
  \begin{align}
    O \left( \left(I_{cont}(\vect{T}) \log T_1 \cdots \log T_{n-1}\right)^{2/(n+1)} \right) \|F\|_{\Gamma}^{2/(n+1)}\|F\|_{1,\infty}^{(n-1)/(n+1)}   .
  \end{align}

  \small 
  \section*{Acknowledgements}

  We thank Alex Kontorovich, Stephen D. Miller, and Siddhartha Sahi for insightful conversations. Moreover we thank Sam Edwards for pointing out his papers on the topic.

  \bibliographystyle{alpha}
  \bibliography{biblio}

    \hrulefill

    \vspace{4mm}
     \noindent Department of Mathematics, Rutgers University, Hill Center - Busch Campus, 110 Frelinghuysen Road, Piscataway, NJ 08854-8019, USA. \emph{E-mail: \textbf{chris.lutsko@rutgers.edu}}

\end{document}